\documentclass[12pt,reqno]{amsart}
\pdfoutput=1

\usepackage[utf8]{inputenc}   
\usepackage[T1]{fontenc}
\usepackage{lmodern}
\usepackage{amsmath}
\usepackage{amsthm}
\usepackage{thmtools}
\usepackage{amssymb}
\usepackage{amscd}
\usepackage[inline]{enumitem}
\usepackage{caption}
\usepackage[dvipsnames]{xcolor}
\usepackage[pdftex, pdfauthor={Dominique Rathel-Fournier},
colorlinks, linkcolor = MidnightBlue, citecolor = WildStrawberry, urlcolor = BrickRed]{hyperref}
\usepackage[margin=1.in]{geometry} 

\usepackage{setspace}
\usepackage{tikz-cd}

\newcommand{\Def}[1]{\textcolor{blue}{#1}}

\newcommand\xqed[1]{%
  \leavevmode\unskip\penalty9999 \hbox{}\nobreak\hfill
  \quad\hbox{#1}}
\newcommand\EndRemark{\xqed{$\triangle$}}

\newcommand{\Q}{\mathbb{Q}}
\newcommand{\Z}{\mathbb{Z}}
\newcommand{\R}{\mathbb{R}}

\newcommand{\C}{\mathbb{C}}

\newcommand{\eps}{\varepsilon}
\newcommand{\vphi}{\varphi}
\newcommand{\wtilde}{\widetilde}

\newcommand{\sSet}{\mathsf{sSet}}	
\newcommand{\Top}{\mathsf{Top}}		
\newcommand{\Arr}{\rightarrow}
\newcommand{\Man}{\mathsf{Man}^q_{\mathrm{dR}}}

\DeclareMathOperator{\colim}{colim}

\newcommand{\Grass}{\mathcal{L}}
\newcommand{\Lag}{\operatorname{lag}}

\newcommand{\Glag}{\Omega^{\Lag}}
\newcommand{\Gform}{\Omega^{\operatorname{form}}}

\DeclareMathOperator{\Th}{Th}
\newcommand{\Sus}{\Sigma}
\newcommand{\Smash}{\wedge}

\newcommand{\RLag}{\mathcal{R}_{\Lag}}

\DeclareMathOperator{\id}{id}

\DeclareMathOperator{\Image}{Image}

\DeclareMathOperator{\Hom}{Hom}
\newcommand{\tens}{\otimes}
\DeclareMathOperator{\Ext}{Ext}
\DeclareMathOperator{\Tor}{Tor}

\newcommand{\bdry}{\partial}

\newcommand{\iso}{\cong}
\newcommand{\del}{\partial}
\DeclareMathOperator{\inc}{inc}
\DeclareMathOperator{\flux}{flux}
\DeclareMathOperator{\stab}{st}
\DeclareMathOperator{\proj}{proj}

\DeclareMathOperator{\Gr}{Gr}

\newcommand{\Per}{\operatorname{Per}_{\omega}}

\newcommand{\pb}{\text{\scriptsize{pb}}}

\newcommand{\hmtpyeq}{\simeq}
\newcommand{\Sect}{\Gamma}
\DeclareMathOperator{\Lifts}{Lifts}

\DeclareMathOperator{\Triv}{triv}
\newcommand{\obs}{\mathfrak{o}}
\DeclareMathOperator{\dR}{dR}

\DeclareMathOperator{\Sing}{Sing}
\newcommand{\SingSm}{\Sing^{\infty}}

\theoremstyle{plain}

\newtheorem{theorem}{Theorem}[section]
\newtheorem{proposition}[theorem]{Proposition}
\newtheorem{lemma}[theorem]{Lemma}
\newtheorem{corollary}[theorem]{Corollary}
\newtheorem{definition}[theorem]{Definition}

\theoremstyle{definition}
\newtheorem{remark}[theorem]{Remark}

\newtheorem{assumption}{Assumption}

\author{Dominique Rathel-Fournier}
\email{d.rathelfournier@gmail.com}
\title{On cobordism groups of Lagrangian immersions}

\begin{document}

\begin{abstract}
We compute the cobordism group $\Glag(M)$ of Lagrangian immersions into a symplectic manifold $(M, \omega)$
in terms of a stable homotopy group of a Thom spectrum constructed from $M$.
This generalizes a result of Eliashberg \cite{Eliashberg} in the case of exact symplectic manifolds.
As applications, we compute $\Glag(M)$ when $M$ is a closed surface and 
give a partial description of $\Glag(M)$ when $M$ is a monotone manifold.
\end{abstract}

\maketitle

\thispagestyle{empty}

\tableofcontents

\section{Introduction}
\label{section:intro}

Let $(M, \omega)$ be a symplectic manifold.
In this work, we
consider the Lagrangian cobordism group $\Glag(M)$ of
Lagrangian immersions of closed manifolds in $M$.
The study of these cobordism groups is a
classical topic in symplectic topology.
They were introduced by Arnold \cite{Arnold},
who computed them for $M = \C$ and $M = T^*S^1$.

In the case of an exact symplectic manifold $(M, d \lambda)$,
the group $\Glag(M)$ was computed by Eliashberg \cite{Eliashberg} 
in terms of the stable homotopy groups of a certain Thom spectrum constructed from $M$.
The key ingredient of the proof is the h-principle for Lagrangian immersions \cite{Gromov-icm, Lees},
which reduces the existence problem for Lagrangian immersions to a purely topological problem.
More detailed computations of the structure of $\Glag(M)$
were carried out by Audin \cite{Audin-calculs, Audin-thesis, Audin-cotangent} in the case of 
$\C^n$ and of cotangent bundles.

The goal of this work is to extend Eliashberg's computation of $\Glag(M)$ to general symplectic manifolds.
As in the exact case, we will show that $\Glag(M)$
is isomorphic to a stable homotopy group of a Thom spectrum constructed from $M$.

\subsection{Main result}

In the remainder of this work, we assume that all manifolds and cobordisms are oriented.
The non-oriented case can be treated by similar methods.

In order to state the main result,
we give a brief description of the relevant spectrum.
A more precise definition is given in Section \ref{section:cob_group_tangential_struct}.
As in the classical setting, this spectrum is the Thom spectrum associated to a
cobordism theory of manifolds equipped with a certain stable tangential structure 
(see Section \ref{subsection:classical_cobordism} for the definitions relevant to cobordism theory).

Inspired by the h-principle, we consider the cobordism theory of manifolds 
equipped with a \emph{stable formal Lagrangian immersion} into $M$.
Here, a formal Lagrangian immersion of a manifold $L$ is a Lagrangian bundle map $TL \to TM$ covering
a map $f:L \to M$ satisfying $[f^*\omega] = 0 \in H^2(L; \R)$,
along with a choice of trivialization of
$f^*\omega$.
A trivialization of $f^*\omega$ can be equivalently seen
as a cohomology class of primitives of $f^*\omega$, or as a homotopy class of lifts of $f$ in the following diagram:
\[
\begin{tikzcd}
											&E^{\omega} \arrow{d}\\
L \arrow{r}[swap]{f} \arrow[dashed]{ur}		&M
\end{tikzcd}
\]
where $E^{\omega} \to M$ is the homotopy fiber of a map $M \to K(\R,2)$ that classifies $[\omega] \in H^2(M; \R)$.
Moreover, \emph{stabilization} of formal solutions corresponds to replacing $TL$ by $\R^k \oplus TL$ and $TM$ by $\C^k \oplus TM$ for large $k$.

In Section \ref{section:cob_group_tangential_struct}, 
we construct a natural classifying space $\Grass_{\infty}^{\omega}M$ for stable formal Lagrangian immersions.
Roughly speaking, $\Grass_{\infty}^{\omega}M$ is obtained from the stable Lagrangian Grassmannian $\Grass_{\infty} M$ of $M$ by 
killing the cohomology class of $\omega$.
The space $\Grass_{\infty}^{\omega}M$ carries a "tautological" stable vector bundle $\gamma^{\omega} \to \Grass_{\infty}^{\omega}M$,
with the property that stable formal Lagrangian immersions of $L$ into $M$
correspond to maps of stable bundles $TL \to \gamma^{\omega}$.
We denote by $-\gamma^{\omega}$ the stable inverse of $\gamma^{\omega}$,
which has the property that a $\gamma^{\omega}$-structure on $TL$ corresponds
to a $-\gamma^{\omega}$-structure on the stable normal bundle $\nu_L$.
We let $\Th(-\gamma^{\omega})$ be the associated Thom spectrum.
The main result of this work is as follows.

\begin{theorem}
\label{thm:main_thm}
For a symplectic manifold $(M, \omega)$ of dimension $2n$, there is a natural isomorphism
\begin{equation}
\label{eq:iso_main_thm}
\Glag(M, \omega) \iso \pi_n \, \Th( -\gamma^{\omega}).
\end{equation}
\end{theorem}

\begin{remark}
The spectrum $\Th(- \gamma^{\omega})$
is canonically defined up to a contractible choice.
Moreover, it is natural in the sense that it
defines a functor from the category of symplectic manifolds (of a given dimension)
and symplectic maps to the stable homotopy category.
The isomorphism \eqref{eq:iso_main_thm} is
a natural isomorphism
of functors.\EndRemark
\end{remark}

\subsection{Computations}
\label{subsection:intro_computations}

As applications of the main theorem,
we compute the group $\Glag(M)$ when $M$ is a closed surface
and investigate the structure of $\Glag(M)$ when $M$ is a monotone symplectic manifold.

\subsubsection{Monotone manifolds}

\begin{definition}
\label{def:monotone_manifold}
A symplectic manifold $(M, \omega)$ is \Def{monotone} if there
is a constant $\tau \in \R$ such that $[\omega] = \tau \, c_1(M, \omega)$
as elements of $H^2(M; \R)$.
\end{definition}

Examples of monotone manifolds include exact manifolds, closed surfaces of genus $g \neq 1$
and $\C P^n$ with the Fubini-Study form.

In Section \ref{subsection:monotone}, we show that the structure
of the spectrum $\Th( -\gamma^{\omega})$
is significantly simpler in the case of monotone manifolds.
To state the result, let $\gamma \to \Grass_{\infty}M$
be the stable vector bundle over the stable Lagrangian Grassmannian $\Grass_{\infty}M$
that classifies stable Lagrangian bundle maps into $TM$ (see Section \ref{subsection:def_gamma} for the definition).
There is a canonical map $\gamma^{\omega} \to \gamma$.

\begin{proposition}
\label{prop:monotone_case}
If $(M, \omega)$ is monotone, then there is a (non-canonical)
equivalence
\begin{equation}
\label{eq:equivalence_monotone_case}
\Th(-\gamma^{\omega})  \hmtpyeq \Th(-\gamma) \Smash K(\R,1)_+.
\end{equation}
\end{proposition}

\begin{remark}
As the proof of Proposition \ref{prop:monotone_case} will show, 
the different choices for the equivalence \eqref{eq:equivalence_monotone_case} are parametrized (up to homotopy)
by choices of cohomology classes of primitives of the pullback of $\omega$ to $\Grass_{\infty}M$.\EndRemark
\end{remark}

Combining Proposition \ref{prop:monotone_case} and Theorem \ref{thm:main_thm}, we obtain the following corollary.

\begin{corollary}
If $(M,\omega)$ is a monotone manifold of dimension $2n$, there is an isomorphism
\begin{equation}
\label{eq:cob_group_monotone_case}
\Glag(M, \omega) \iso H_n( K(\R,1) ; \Th(-\gamma) ) 
\end{equation}
\end{corollary}
The right-hand side of Equation \eqref{eq:cob_group_monotone_case}
can be described further using the Atiyah-Hirzebruch spectral sequence;
see Proposition \ref{prop:iso_tensor_product}.

\subsubsection{Surfaces}

Let $\Sigma$ be a closed symplectic surface.
We denote by $S\Sigma$ the unit tangent bundle of $\Sigma$.
Note that $S\Sigma$ is isomorphic to the oriented Lagrangian Grassmannian of $\Sigma$.

\begin{theorem}
\label{thm:computation_surface}
Let $\Sigma$ be a closed symplectic surface of genus $g$.
Then 
\[
	\Glag(\Sigma) \iso 
	\begin{cases}
						H_1(S\Sigma;\Z) \oplus \R 		&\text{if } g \neq 1, \\
						H_1(S\Sigma;\Z) \oplus \R/\Z 	&\text{if } g = 1.
	\end{cases}
\]
\end{theorem}

\begin{remark}
In the case of surfaces of genus $g \geq 2$, the group $\Glag(\Sigma)$ was computed
in \cite{DRF23}, using previous work of Perrier \cite{Perrier19}.
These computations are based on geometric methods; the present work provides a homotopy-theoretic
proof of this result.\EndRemark
\end{remark}

\subsection{Outline of the proof and organization of the paper}

The isomorphism \eqref{eq:iso_main_thm} of Theorem \ref{thm:main_thm} 
is obtained in three steps, which are illustrated in the following diagram:
\begin{equation}
\label{cd:outline_of_proof}
\begin{tikzcd}[column sep = large, labels = { inner sep = 1ex }]
\Glag(M) \arrow{r}{\iso}[swap]{\text{(i)}}
&\Gform(M) \arrow{r}{\iso}[swap]{\text{(ii)}}
&\Omega_n(\gamma^{\omega}) \arrow{r}{\iso}[swap]{\text{(iii)}}
&\pi_n \Th(-\gamma^{\omega}).
\end{tikzcd}
\end{equation}
We give a brief outline of each step in Diagram \eqref{cd:outline_of_proof}.
Step (i) is carried out in Section \ref{section:formal_cob_group}, where
it is shown that $\Glag(M)$ is isomorphic to 
a cobordism group $\Gform(M)$ of formal Lagrangian immersions. The proof is an application
of the h-principle for Lagrangian immersions.
The stable vector bundle $\gamma^{\omega}$ which classifies stable formal Lagrangian immersions $L \to M$
is constructed in Section \ref{section:cob_group_tangential_struct}.
The isomorphism (ii) is then obtained in Section \ref{section:proof_main_thm}
by a stabilization argument.
Finally, the isomorphism (iii) is obtained in Section \ref{section:proof_main_thm} by an application of the Pontrjagin-Thom Theorem.

The rest of this paper is organized as follows.
Section \ref{section:computations} contains the computations whose results are stated in Section \ref{subsection:intro_computations}.
Section \ref{section:prelim} collects some preliminaries related to Lagrangian immersions
and classical cobordism theory.
The appendix contains the construction,
for a manifold $M$ equipped with a closed form $\omega$,
of a fibration $E^{\omega} \to M$
that represents the homotopy fiber of a map $M \to K(\R,q)$ classifying $[\omega] \in H^q(M; \R)$.
The main feature of this construction is that it is functorial as a function of the pair $(M, \omega)$,
which in turn allows for a functorial definition of the spectrum $\Th(-\gamma^{\omega})$ of Theorem \ref{thm:main_thm}.

\subsection{Acknowledgements}

This work is part of the author's PhD thesis at the University of Montreal under the supervision of Octav Cornea and François Lalonde.
The author would like to thank Octav Cornea for his encouragements and guidance.
The author would also like to thank Jordan Payette for helpful discussions about h-principles.
The author acknowledges the financial support of NSERC Grant \#{}504711 and FRQNT Grant \#{}300576.

\section{Preliminaries}
\label{section:prelim}

\subsection{Cobordism groups of Lagrangian immersions}

Let $(M, \omega)$ be a symplectic manifold of dimension $2n$.

\begin{definition}
An immersion $f: L \to M$ of a manifold $L$ of dimension $n$ is \Def{Lagrangian} if $f^*\omega = 0$.
\end{definition}

We recall the definition of a Lagrangian cobordism between two Lagrangian immersions,
which is due to Arnold \cite{Arnold}.
Let $f_0 : L_0 \to M$
and $f_1: L_1 \to M$
be Lagrangian immersions of (possibly disconnected) closed manifolds into $M$.
\begin{definition}
A \Def{Lagrangian cobordism} from $f_0$ to $f_1$ consists
of a compact manifold $V$ with boundary $\bdry V = L_0 \sqcup L_1$, 
along with
a Lagrangian immersion $F: V \to T^*[0,1] \times M$ satisfying the following assumptions:
\begin{itemize}
	\item $F^{-1}( T^*_k[0,1] \times M) = L_k$ for $k=0,1$.
	\item There is a collar embedding $j_0: L_0 \times [0, \eps) \to V$
	such that $F \circ j_0(x,t) = (t, 0 , f_0(x))$. 
	Similarly, there is a collar embedding $j_1: L_1 \times (1- \eps, 1] \to V$
	such that $F \circ j_1(x,t) = (t, 0 , f_1(x))$.
\end{itemize}
\end{definition}

One obtains variations on this definition by requiring that the manifolds be equipped
with extra structures.
In this work, we will only consider the \emph{oriented} theory, i.e. we shall assume
that the manifolds $L_i$ are oriented and that $V$ is an oriented cobordism.
The unoriented case can be treated by similar methods.

Lagrangian cobordism defines an equivalence relation on Lagrangian immersions.
We let $\Glag(M)$ be the associated Lagrangian cobordism group,
where the group operation is disjoint union.
Note that the inverse in $\Glag(M)$
corresponds to reversing the orientation.

If $f: (M, \omega) \to (M', \omega')$ is a symplectic map between manifolds of the same dimension,
then there is an induced pushforward morphism $f_*: \Glag(M) \to \Glag(M')$.
In this way, the Lagrangian cobordism group defines a functor
on the category of symplectic manifolds of a given dimension.

\subsection{The h-principle for Lagrangian immersions}
\label{subsection:h-principle}

Let $L$ be a manifold.
A map of vector bundles $\vphi: TL \to TM$
is \Def{Lagrangian} if $\vphi(T_x L)$ is a Lagrangian subspace for every $x \in L$.
Equivalently, a Lagrangian bundle map
is a section of the differential relation $\RLag \subset J^1(L,M)$
of Lagrangian immersions, where $J^1(L,M)$ denotes the 1st order jet space.\footnote{Sections of $\RLag$ are often called \emph{formal Lagrangian immersions}, but we shall reserve this name for the objects introduced in Definition \ref{def:formal_Lag_immersion} below, which
carry the additional datum of a trivialization of the pullback of $\omega$.}
A Lagrangian bundle map is \Def{holonomic} if it is equal to $Df:TL \to TM$ for some Lagrangian immersion $f: L \to M$.

We shall need the following form of the h-principle for Lagrangian immersions, due to Gromov \cite{Gromov-icm,Gromov-pdr} and Lees \cite{Lees}.
See also \cite{Eliashberg-Mishachev} for a textbook account.

\begin{theorem}
\label{thm:h-principle}
Let $\vphi: TL \to TM$ be a Lagrangian bundle map covering a map $f: L \to M$.
If $[f^*\omega] = 0 \in H^2(L; \R)$,
then $\vphi$ is homotopic to a holonomic solution.
Moreover, 
if $\vphi$ is holonomic on a neighborhood of a closed subset $A \subset L$,
then $\vphi$ is homotopic rel. $A$ to a holonomic solution
provided that $[f^*\omega] = 0 \in H^2(L, A; \R)$.
\end{theorem}

We shall also need the following special case of the parametric h-principle for exact Lagrangian immersions.
Let $(M, \omega)$ be an exact symplectic manifold.
Given a $1$-form $\lambda$ with $d \lambda= \omega$, a Lagrangian immersion $f:L \to M$
is called \Def{$\lambda$-exact} if $f^*\lambda$ is exact.

\begin{theorem}[{\cite[16.3.1]{Eliashberg-Mishachev}}]
\label{thm:h-principle_exact}
Let $(M, \omega)$ be an exact symplectic manifold and $\lambda_t$
be a smooth path of $1$-forms with $d \lambda_t = \omega$ for all $t \in [0,1]$.
Let $f_t: L \to M$ be a path of Lagrangian immersions
such that $f_0$ is $\lambda_0$-exact.
Then there exists a path of Lagrangian immersions $(g_t)$
such that $g_0 = f_0$ and $g_t$ is $\lambda_t$-exact for all $t \in [0,1]$.
\end{theorem}

\subsection{Cobordism of manifolds with stable tangential structures}
\label{subsection:classical_cobordism}

In this section, we recall some well-known notions from classical cobordism theory.
The original reference for this material is the work of Lashof \cite{Lashof}; see also Stong \cite{Stong} for
a detailed account. 

\subsubsection{Stable vector bundles and Thom spectra}

We shall consider cobordism theories of manifolds equipped
with a certain type of \emph{stable tangential structure},
which is classified by
a stable vector bundle.
For our purposes, it will be convenient to think of stable vector bundles 
in terms of classifying maps. Moreover, we shall only consider oriented bundles.
Therefore, we define a stable vector bundle
$\xi$ over a space $B$ to be a map $\xi: B \to BSO$,
where $BSO = \colim_n BSO(n)$.
The category of stable vector bundles is the category of spaces over $BSO$.
Unless specified otherwise, we will assume
that $\xi$ is a Serre fibration.\footnote{This means that we restrict to fibrant objects
of the Quillen model structure of the category of spaces over $BSO$.}

A stable vector bundle has an associated Thom spectrum defined as follows.
For each $r \geq 1$, form the following pullback diagram:
\[
\begin{tikzcd}
B_r \arrow{r}\arrow{d}[swap]{\xi_r} \arrow[dr, phantom, "\pb"]	&B \arrow{d}{\xi} \\
BSO(r) \arrow{r}[swap]{j_r}										&BSO
\end{tikzcd}
\]
where $j_r: BSO(r) \to BSO$ is the inclusion. 
We then have induced pullback diagrams
\begin{equation}
\label{eq:stab_maps_bases}
\begin{tikzcd}
B_r \arrow{r}{i_r}\arrow{d}[swap]{\xi_r} \arrow[dr, phantom, "\pb"]		&B_{r+1} \arrow{d}{\xi_{r+1}} \\
BSO(r) \arrow{r}														&BSO(r+1).
\end{tikzcd}
\end{equation}
Let $\gamma_r \to B_r$ be
the pullback of the universal vector bundle over $BSO(r)$.
Then there is a map of vector bundles induced by \eqref{eq:stab_maps_bases}:
\begin{equation}
\label{eq:stab_maps_bundles}
\begin{tikzcd}
\R \oplus \gamma_r \arrow{r}\arrow{d}	&\gamma_{r+1} \arrow{d} \\
B_r \arrow{r}{i_r}						&B_{r+1}
\end{tikzcd}
\end{equation}
which is fiberwise an orientation-preserving isomorphism.
We define $\Th(\xi)$ to be the Thom spectrum whose $r$-th space 
is the Thom space $\Th(\gamma_r)$
and whose structure maps $\Sus\Th( \gamma_r ) \to \Th(\gamma_{r+1})$
are induced by Diagram \eqref{eq:stab_maps_bundles}.

\subsubsection{Cobordism of manifolds with $\xi$-structures}

Let $\xi: B \to BSO$ be a stable vector bundle.
If $L$ is an oriented smooth manifold, 
then a tangential $\xi$-structure on $L$
consists of a choice of classifying map $\kappa: L \to BSO$ for the stable tangent bundle of $L$ and 
a vertical homotopy class of lifts of $\kappa$ in the following diagram:
\[
\begin{tikzcd}
													&B \arrow{d}{\xi} \\
L \arrow{r}[swap]{\kappa} \arrow[dashed]{ur}{\ell}	&BSO
\end{tikzcd}
\]
Two such structures are considered equivalent if they are concordant,
in the sense that there is a $\xi$-structure
on $L \times [0,1]$
that restricts to the given structures over the boundary.
Normal $\xi$-structures on $L$
are defined in a similar way by considering classifying maps for the stable normal bundle $\nu_L$.

One can then define a cobordism
theory of manifolds equipped with tangential or normal $\xi$-structures;
see \cite{Lashof, Stong}.
We denote the associated cobordism groups by $\Omega_n(\xi)$ for the tangential version,
and by $\Omega_n^{\nu}(\xi)$ for the normal version.

We shall need the following version of the Thom-Pontrjagin Theorem 
for stable vector bundles, originally due to Lashof \cite{Lashof}.
\begin{theorem}[\cite{Lashof}]
\label{thm:PT}
For a stable vector bundle $\xi$,
there is an isomorphism
\[
\Omega_n^{\nu}(\xi) \iso \pi_n \Th(\xi).
\]
\end{theorem}

The cobordism groups that appear naturally in this work are
cobordism groups of tangential structures.
To apply Theorem \ref{thm:PT} to this case, we 
shall make use of the following duality between tangential and normal structures.

A stable vector bundle $\xi$
has a stable inverse $-\xi$ defined as follows.
Consider the map $\iota: BSO \to BSO$ that is
induced in the colimit by the maps
\[
\Gr_n( \R^k) \to \Gr_{k-n}( \R^k)
\]
which send a subspace to its orthogonal.
We then define $-\xi: B \to BSO$ as the pullback fibration $\iota^* \xi$.
For any manifold $L$, there is a canonical bijection between
tangential $\xi$-structures on $L$
and normal $(-\xi)$-structures on $L$.
This bijection induces an isomorphism $\Omega_n(\xi) \iso \Omega^{\nu}_n(-\xi)$.
This yields the following reformulation of Theorem \ref{thm:PT}
in terms of cobordism groups of tangential structures.

\begin{theorem}
\label{thm:PT_tangential}
For a stable vector bundle $\xi$,
there is an isomorphism
\[
\Omega_n(\xi) \iso \pi_n \Th(-\xi).
\]
\end{theorem}

\subsection{Stable vector bundles from sequences of vector bundles}
\label{prelim:construction_stable_bundle}

The stable vector bundles considered in this work will be constructed
by the following procedure.
Suppose that we are given a sequence of oriented vector bundles $\gamma_r \to B_r$
of rank $r$, along with stabilization maps $\R \oplus \gamma_r \to \gamma_{r+1}$ covering maps $i_r: B_r \to B_{r+1}$
as in Diagram \ref{eq:stab_maps_bundles}.
We assume that the maps $i_r$ are cofibrations.
Define $B = \colim_r B_r$.

Furthermore, suppose that we have chosen classifying maps
\[
\begin{tikzcd}
\gamma_r \arrow{d} \arrow{r}	& U_r \arrow{d} \\
B_r \arrow{r}					& BSO(r)
\end{tikzcd}
\]
for the bundles $\gamma_r$, where $U_r \to BSO(r)$ is the universal bundle.
We then obtain a stable vector bundle $\xi: B \to BSO$
by taking the map induced in the colimit by the classifying maps $B_r \to BSO(r)$.

This construction has the following feature.
For a manifold $L$ of dimension $n$, the stabilization maps $\R \oplus \gamma_r \to \gamma_{r+1}$
induce stabilization maps
\[
\begin{tikzcd}
\left[\R^k \oplus TL  , \gamma_{n+k}\right] \arrow{r} &\left[ \R^{k+1} \oplus TL,  \gamma_{n+k+1} \right]
\end{tikzcd}
\]
where for two bundles $U$ and $V$, we denote by $\left[ U,V\right]$ the set of homotopy classes of 
bundle maps $U \to V$
that are fiberwise orientation-preserving isomorphisms.
There is then a canonical bijection between the set of $\xi$-structures on $L$ and the 
colimit
\[
\left[ TL, \gamma \right] := \colim_{k} \left[ \R^k \oplus TL , \gamma_{n+k}\right]
\]

\section{Cobordism groups of formal Lagrangian immersions}
\label{section:formal_cob_group}

In this section, we use the h-principle for Lagrangian immersions
to show that $\Glag(M)$ is isomorphic
to a cobordism group $\Gform(M)$
of formal Lagrangian immersions.
Inspired by the h-principle, the group $\Gform(M)$ that
we consider
has generators given by Lagrangian bundle maps $TL \to TM$
equipped with a cohomology class of primitives of the pullback of $\omega$,
which we call a \emph{trivialization}.
In Section \ref{subsection:triv_exact_forms}, we introduce trivializations
of exact forms
and prove some basic properties.
In Section \ref{subsection:cob_group_formal_lags}, 
we introduce the cobordism group $\Gform(M)$
and prove the isomorphism $\Glag(M) \iso \Gform(M)$.

\subsection{Trivializations of exact differential forms}
\label{subsection:triv_exact_forms}

Let $L$ be a manifold, possibly with boundary.
Let $\alpha$ be an exact $q$-form on $L$.
A primitive of $\alpha$ is a form $\lambda$ of degree $q-1$
with $d \lambda = \alpha$.
Two primitives $\lambda$ and $\lambda'$
of $\alpha$ are called cohomologous
if the closed form $\lambda - \lambda'$ is exact.

\begin{definition}
\label{def:triv_exact_form}
A \Def{trivialization} of an exact form $\alpha$ is a cohomology class of primitives of $\alpha$.
The set of trivializations of $\alpha$ is denoted $\Triv(\alpha)$.
\end{definition}

Note that if $\alpha = 0$, then $\Triv(\alpha) = H_{\dR}^{q-1}(L)$.
In this case, there is a preferred trivialization given by $0 \in H_{\dR}^{q-1}(L)$.
In the general case, $\Triv(\alpha)$ is an affine space 
over $H_{\dR}^{q-1}(L)$,
but there is no canonical identification with $H_{\dR}^{q-1}(L)$.

The set of trivializations defines a functor
in the following way:
if $(L, \alpha)$ and $(L', \alpha')$
are manifolds equipped with exact $q$-forms
and $f: L \to L'$ is a smooth map such that $f^*\alpha' = \alpha$,
then there is a natural pullback map $f^*: \Triv(\alpha') \to \Triv(\alpha)$
which is equivariant with respect to the actions of $H^{q-1}_{\dR}(L)$ and 
$H^{q-1}_{\dR}(L')$.

\subsubsection{Extension of trivializations to a cobordism.}

Suppose now that $V$ is a compact manifold with boundary, $\alpha$ is an exact $q$-form on $V$,
and we are given a trivialization $\tau$ of the restriction $\alpha|_{\bdry V}$.
The obstruction to the existence of a trivialization on $V$ 
extending $\tau$ is given by a class $\obs(\alpha; \tau) \in H^{q}(V, \bdry V; \R)$.
This class can be described in the de Rham model as follows.
Following \cite[p.78]{Bott-Tu}, 
we take as a model for relative cohomology the mapping fiber of the map induced by the inclusion $\iota: \bdry V \to V$:
\[
\Omega^*(V, \bdry V) := \Omega^*(V) \oplus \Omega^{*-1}( \bdry V),
\]
with the differential given by $d(\eta, \lambda) = (d \eta, \iota^*\eta- d \lambda)$.
A cocycle of $\Omega^*(V, \bdry V)$ is then a pair $(\eta, \lambda)$
such that $d \eta = 0$ and $d\lambda = \eta|_{\bdry V}$.
Note that the complex $\Omega^*(V, \bdry V)$ is a shift of the mapping cone of $\iota^*$.

Suppose now that $\alpha$ is an exact $q$-form on $V$
and that we are given a trivialization $\tau$ of $\alpha|_{\bdry V}$.
Then any representative $\lambda$ of the trivialization
determines a cocycle $(\alpha, \lambda) \in \Omega^q(V, \bdry V)$, and
the cohomology class $[(\alpha, \lambda)] \in H_{\dR}^q(V, \bdry V)$
does not depend on the choice of $\lambda$.
We then set $\obs(\alpha; \tau) = [(\alpha, \lambda)]$.
Note that by definition $\obs(\alpha; \tau) = 0$ if and only if there is a 
trivialization of $\alpha$ over $V$ that extends the trivialization $\tau$ over $\bdry V$.

\subsubsection{Homotopies and flux.}

We now consider the following situation.
Let $(M, \omega)$ be a manifold equipped with a closed $q$-form $\omega$.
Let $H: L \times I \to M$ be a homotopy, where $L$ is a closed manifold,
and suppose that $H^*\omega$ is an exact form.
For convenience, we will use the notation $V = L \times I$ and write $i_t: L \to V$ for the inclusion at time $t$
and $f_t = H \circ i_t$.

The homotopy $H$ induces
an $H_{\dR}^{q-1}(L)$-equivariant isomorphism $\psi_H: \Triv(f_0^* \omega) \to \Triv(f_1^* \omega)$ as follows.
Let $\lambda_0 \in \Omega^{q-1}(L)$ be a form with $d\lambda_0 = f_0^*\omega$.
By the homotopy invariance of de Rham cohomology, there exists a primitive
$\Lambda$ of $H^*\omega$
with $i_0^*\Lambda = \lambda_0$.
Moreover, any two such primitives are cohomologous.
We then set $\Psi_H(\lambda_0) = i_1^*\Lambda$.
It is easy to check that this induces a well-defined map on trivializations.

Suppose now that $\tau$ is a trivialization of $H^*\omega|_{\bdry V}$.
Note that this is the same as a pair of trivializations $\tau_k$ of $f_k^*\omega$ for $k=0,1$.
We define the flux of $H$
with respect to $\tau$
as the unique class $\flux(H; \tau) \in H_{\dR}^{q-1}(L)$
such that $\psi_H(\tau_0) = \tau_1 + \flux(H; \tau)$.
Note that in the case where $f_k^* \omega =0$ and both ends are equipped with the same trivialization,
this agrees with the standard definition of flux as the $\omega$-volume
swept out by cycles under the homotopy.

We collect the properties of flux that we shall need in the following lemma.
In the statement of the lemma, we shall make use of the isomorphism $H_{\dR}^{q-1}(L) \iso H_{\dR}^q(V, \bdry V)$
given by the following composition:
\begin{equation}
\label{eq:iso_relative_coh}
\begin{tikzcd}
H_{\dR}^{q-1}(L) \arrow{r}{j} &H_{\dR}^{q-1}(\bdry V) \arrow{r}{\delta} & H_{\dR}^q(V, \bdry V)
\end{tikzcd}
\end{equation}
where the map $j$ is the inclusion of the component $L \times \{ 1 \} \subset \bdry V$, 
and $\delta$ is the boundary map of the long exact sequence of the pair.

\begin{lemma}
\label{lemma:properties_of_flux}
\hfill
\begin{enumerate}[label = (\roman*), font=\normalfont]

	\item $\flux(H; \tau)$ is invariant under homotopies of $H$ relative to the boundary.
	
	\item If $H$ and $K$ are composable homotopies 
	and $H\# K$ denotes their concatenation, then
	\[
	\flux( H \# K; \tau) = \flux(H; \tau_-) + \flux(K; \tau_+),
	\]
	where $\tau_{\pm}$ are trivializations
	such that $\tau_-|_{L_0} = \tau|_{L_0}$, $\tau_+|_{L_1} = \tau|_{L_1}$ and $\tau_-|_{L_1} = \tau_+|_{L_0}$.
	
	\item The isomorphism \eqref{eq:iso_relative_coh}
	sends $\flux(H; \tau)$ to $\obs(H^*\omega; \tau)$.
\end{enumerate}
\end{lemma}
\begin{proof}
These properties follow directly from the definition of flux.
\end{proof}

\subsection{Cobordism of formal Lagrangian immersions}
\label{subsection:cob_group_formal_lags}

Let $(M, \omega)$ be a symplectic manifold of dimension $2n$.
We consider the cobordism theory of Lagrangian bundle maps $TL \to TM$ equipped 
with a trivialization of the pullback of $\omega$.

\begin{definition}
\label{def:formal_Lag_immersion}
Let $L$ be a closed manifold of dimension $n$.
A \Def{formal Lagrangian immersion} of $L$ in $M$ consists of
\begin{enumerate}[label = (\roman*), font=\normalfont]
	\item a Lagrangian bundle map $\vphi: TL \to TM$
covering a map $f: L \to M$ that satisfies $[f^*\omega] = 0 \in H^2_{\dR}(L)$,
	\item a choice of trivialization of $f^*\omega$.
\end{enumerate}
\end{definition}

\begin{remark}
Suppose that $(M,\omega)$ is an exact symplectic manifold
and that we fix a choice of primitive $\lambda$ of $\omega$.
Then for any map $f: L \to M$,
the choice of $\lambda$
determines an identification
$\Triv(f^*\omega) \iso H^1_{\dR}(L)$.
Thus, in this case we could
replace the datum (ii) in Definition \ref{def:formal_Lag_immersion}
by a choice of cohomology class $a \in H^1(L;\R)$.
This is the approach taken in \cite{Eliashberg}.
Note however that this reformulation of the definition depends on the specific choice of $\lambda$;
in particular, it breaks the functoriality with respect to symplectic maps.\EndRemark
\end{remark}

In order to define cobordisms between formal Lagrangian immersions, we make the following definitions.
First, given a Lagrangian bundle map $\vphi: TL \to TM$,
we define its stabilization to be the Lagrangian bundle map $\stab \vphi: \R \oplus TL \to \C \oplus TM$
obtained by taking the direct sum with the canonical inclusion $\R \subset \C$.

Secondly, if $V$ is an oriented cobordism with negative end $L_0$ and positive end $L_1$,
then there are trivializations 
$TV|_{L_k} \iso \R \oplus T L_k$, canonical up to a contractible choice,
which are determined by the outward pointing vectors for the positive end and
by the inward pointing vectors for the negative end.
In particular, suppose that we have a Lagrangian bundle map $\Phi: TV \to \C \oplus TM$
covering a map $F: V \to M$, with $f_k = F|_{L_k}$.
Then $\Phi$
induces Lagrangian bundle maps $\Phi|_{L_k}: \R \oplus T L_k \to \C \oplus TM$
covering $f_k$,
well-defined up to bundle homotopy over $f_k$ (in the sense that the homotopies
are required to project to $f_k$ for all time).

\begin{definition}
\label{def:formal_Lag_cob}
Let $\vphi_0: TL_0 \to TM$ and $\vphi_1: TL_1 \to TM$
be formal Lagrangian immersions with base maps $f_k: L_k \to M$.
A \Def{formal Lagrangian cobordism}
from $\vphi_0$ to $\vphi_1$ consists of an oriented cobordism
$V$ from $L_0$ to $L_1$,
along with a Lagrangian bundle map $\Phi: V \to \C \oplus TM$
covering a map $F: V \to M$,
which satisfies the following conditions:

\begin{itemize}
	\item $F|_{L_k} = f_k$ for $k=0,1$.
	\item $\Phi|_{L_k}$ and $\stab \vphi_k$ are homotopic over $f_k$ for $k= 0,1$.

	\item $F^*\omega$ is exact and $V$ is equipped with
	a trivialization of $F^*\omega$
	that restricts to the given trivializations of $f_k^*\omega$ over $L_k$.
\end{itemize}
\end{definition}

We denote by $\Gform(M)$ the cobordism group of formal Lagrangian immersions
modulo formal Lagrangian cobordisms.
This is clearly functorial with respect to symplectic maps $(M, \omega) \to (M', \omega')$
between manifolds of the same dimension.

There is a natural morphism $\psi: \Glag(M) \to \Gform(M)$
that sends a Lagrangian immersion $f: L \to M$ to the formal Lagrangian $Df: TL \to TM$ equipped with the zero trivialization.
The main result of this section is as follows.

\begin{proposition}
\label{prop:iso_formal_cob_group}
$\psi$ is a an isomorphism.
\end{proposition}

For the proof, we shall need the following lemma.

\begin{lemma}
\label{lemma:existence_of_homotopy_of_arbitrary_flux}
Let $f: L \to M$ be a Lagrangian immersion.
For every $a \in H^1(L;\R)$,
there exists a path of Lagrangian immersions $F = (f_t)_{t \in [0,1]}$
such that $f_0 = f$ and $\flux(F) = a$.
\end{lemma}
\begin{proof}
By the Weinstein neighborhood theorem, it suffices
to prove the lemma in the case of the Lagrangian immersion $\iota: L \to U$,
where $U$ is a tubular neighborhood of the zero-section in $T^*L$ and $\iota$ is the zero-section.
Let $\alpha_t$ be a smooth family of closed $1$-forms on $U$
such that $\alpha_0 = 0$ and $[\iota^*\alpha_1] = a$.
Let $\lambda$ be the Liouville form on $U$ and consider
the family of primitives $\lambda_t = \lambda - \alpha_t$
of the standard symplectic form.
Consider the path of Lagrangian immersions $(g_t)_{t \in [0,1]}$ given by $g_t = \iota$ for all $t$.
Note that $g_0 = \iota$ is exact with respect to $\lambda_0 = \lambda$.
Applying the parametric h-principle for exact Lagrangian immersions (Theorem \ref{thm:h-principle_exact}) to the 
path $(g_t)$ and the family $(\lambda_t)$, 
we obtain a path of Lagrangian immersions $(f_t)_{t \in [0,1]}$ such
that $f_0 = \iota$ and $f_t^* \lambda_t$ is exact for all $t$.
The path $(f_t)$ has flux $[f_1^*\lambda] - [\iota^*\lambda] = a$, as required.
\end{proof}

\begin{proof}[Proof of Proposition \ref{prop:iso_formal_cob_group}]




First, we show that $\psi$ is injective.
Suppose that $f: L \to M$ is a Lagrangian immersion such that 
$[(D f, 0)] = 0$ in $\Gform(M)$.
This means that there is a cobordism $V$ with $\bdry V = L$
and a Lagrangian bundle map $\Phi: TV \to \C \oplus TM $ 
such that $\Phi|_L$ is homotopic to $\stab Df$ over $f$.
Moreover, denoting by $F: V \to M$ the base map of $\Phi$, 
there is a trivialization of $F^*\omega$ that
extends the zero trivialization on $\bdry V$.
In particular, this implies that $[F^*\omega] =0 \in H^2(V, \bdry V; \R)$.

By a homotopy, we may assume that,
in a collar neighborhood $U \iso L \times [0,\eps)$ of $\bdry V$,
we have $F(x,s) = f(x)$ and $\Phi(x,s)= \stab Df_x$,
where $(x,s) \in L \times [0, \eps)$ are collar coordinates.

Let $\rho: V \to [0,1]$ be a smooth map such that $\rho^{-1}(0)= \bdry V$
and $\rho$ extends the collar coordinate $s$.
Let $\wtilde{F}: V \to \C \times M$ be the map defined by
$\wtilde{F}(p) = (\rho(p), F(p))$.
Then $\Phi$ lifts in the obvious way to a Lagrangian bundle map $\wtilde{\Phi}: TV \to T( \C \times M)$
covering $\wtilde{F}$.
Note that this is holonomic on the collar neighborhood $U$.
Moreover, we have that 
\[
\left[\wtilde{F}^*( \omega_{\C} \oplus \omega) \right] = \left[ F^* \omega \right] = 0 \in H^2(V, \bdry V; \R),
\]
hence the cohomological
condition of the relative $h$-principle (Theorem \ref{thm:h-principle}) is satisfied.
By the relative $h$-principle, $\wtilde{F}$ is homotopic rel. boundary to a true Lagrangian cobordism,
hence $[f] = 0$ in $\Glag(M)$.




We now prove that $\psi$ is surjective.
Let $\vphi: TL \to TM$ be a formal Lagrangian covering a map $f: L \to M$,
equipped with a trivialization $\tau$ of $f^*\omega$.
Since $f^*\omega$ is exact, by the $h$-principle (Theorem \ref{thm:h-principle})
$\vphi$ is homotopic 
through Lagrangian bundle maps
to $Dg$, where $g: L \to M$ 
is a true Lagrangian immersion.
Let $\wtilde{H}: TL \times I \to TM$ be a homotopy from $\vphi$ to $D g$ 
covering a homotopy $H: L \times I \to M$ from $f$ to $g$.
Then $\wtilde{H}$ has an obvious lift to a Lagrangian bundle map $\Phi: T(L \times I) \to \C \oplus TM$
by sending $\del/\del t \in T_s I$ to $1 \in \C$.

The form $H^*\omega$ is exact by the homotopy invariance of de Rham cohomology.
By Lemma \ref{lemma:existence_of_homotopy_of_arbitrary_flux} 
and Lemma \ref{lemma:properties_of_flux}(ii), 
up to changing $g$ by a Lagrangian homotopy
we may assume that the homotopy $H$ has flux $0$ with respect to the trivialization $\tau$ over $L \times \{ 0 \}$
and the zero trivialization over $L \times \{ 1 \}$.
By Lemma \ref{lemma:properties_of_flux}(iii), 
there is a trivialization of $H^* \omega$ over $L \times I$ that extends the given trivializations on the ends.
Hence $\Phi$ equipped with a choice of trivialization is a formal Lagrangian cobordism, which shows that $[(\vphi, \tau)] = [(Dg, 0)]$ in $\Gform(M)$.
We conclude that $\psi$ is surjective.
\end{proof}

\section{Formal Lagrangian immersions as tangential structures}
\label{section:cob_group_tangential_struct}

In this section, we construct for every symplectic manifold $M$ a 
natural stable vector bundle $\gamma^{\omega}$
with the property that
$\gamma^{\omega}$-structures on the tangent bundle of a manifold $L$
correspond
to homotopy classes of stable formal Lagrangian immersions of $L$ in $M$.

\subsection{Lagrangian Grassmannians}

For a symplectic vector bundle $E \to B$ of rank $2k$,
we denote by $\Grass(E)$ the Grassmannian of oriented Lagrangian subspaces of $E$.
The projection $\Grass(E) \to B$ is a fiber bundle with fiber $\Lambda_k$, the Grassmannian of oriented Lagrangian subspaces of
the standard symplectic vector space $\R^{2k}$.
Recall that there is a transitive action of $U(k)$ on $\Lambda_k$, which induces a $U(k)$-equivariant diffeomorphism 
$\Lambda_k \iso U(k)/SO(k)$.

The space $\Grass(E)$ carries a tautological bundle $\gamma \to \Grass(E)$, which
is canonically oriented.
There is a tautological Lagrangian bundle map $\gamma \to E$
covering the projection $\Grass(E) \to B$.
The map $\gamma \to E$ has the following universal property.
\begin{lemma}
\label{lemma:univ_prop_lag_grass}
Let $\eta$ be an oriented vector bundle of rank $k$.
Then any Lagrangian bundle map $\vphi: \eta \to E$
has a unique factorization as follows:
\[
\begin{tikzcd}
												&\gamma \arrow{d} \\
\eta \arrow{r}[swap]{\vphi} \arrow{ur}{\Phi}	&E
\end{tikzcd}
\]
where $\Phi$ is fiberwise an orientation-preserving isomorphism.
\end{lemma}

As a consequence of the universal property, there
is a canonical bijection between Lagrangian bundle maps $\eta \to E$
and bundle maps $\eta \to \gamma$ that are fiberwise orientation-preserving isomorphisms.
In this way, we may see
$\Grass(E)$ as a classifying space for Lagrangian bundle maps of oriented bundles into $E$.

\subsubsection{Stabilization.}

There is a stabilization operation on symplectic bundles that sends $E$ to $\C \oplus E$,
where $\C$ denotes the trivial bundle with the standard symplectic structure, and the
direct sum is equipped with the product structure.
There is an induced stabilization map $j: \Grass(E) \to \Grass(\C \oplus E)$
that sends an oriented Lagrangian subspace $V \subset E$ to $\R \oplus V \subset \C \oplus E$,
where $\R$ has the standard orientation.
We shall need the following stability lemma.

\begin{lemma}
\label{lemma:stab_homotopy_Lag_grass}
Let $E \to B$ be a symplectic vector bundle of rank $2k$. 
Then the stabilization map $j: \Grass(E) \to \Grass (\C \oplus E)$ is $k$-connected.
Moreover, for $k=1$, $j$ is an isomorphism on $\pi_1$.

\end{lemma}
\begin{proof}
To prove the first statement, it suffices to prove that the embedding of the fibers $\Lambda_k \to \Lambda_{k+1}$ is $k$-connected.
To show this, recall that the embedding
$SO(k) \to SO(k+1)$ is $(k-1)$-connected
and that the embedding $U(k) \to U(k+1)$ is $2k$-connected.
The claim then follows by comparing the
exact sequences in homotopy of
the fiber bundles $SO(r) \hookrightarrow U(r) \to \Lambda_r$ for $r = k$ and $r= k+1$.

In the case $k=1$, we have that $\pi_1 \Lambda_1 \iso \Z \iso \pi_1\Lambda_2$.
Therefore, the stabilization map $\pi_1 \Lambda_1 \to \pi_1\Lambda_2$ is an isomorphism.
As above, this implies that the map $j:\Grass(E) \to \Grass( \C \oplus E)$ is also an isomorphism $\pi_1$.
\end{proof}

For $q \geq k$, we consider the stabilization $\C^{q-k} \oplus E$ of rank $2q$
and its Lagrangian Grassmannian $\Grass_q(E) = \Grass( \C^{q-k} \oplus E)$.
Let $\Grass_{\infty}(E)$ be the colimit
of the spaces $\Grass_{q}(E)$ along the stabilization maps.
Lemma \ref{lemma:stab_homotopy_Lag_grass} then yields the following stability property.
\begin{corollary}
\label{cor:inclusion_in_colimit_is_k_connected}
For a symplectic bundle $E$ of rank $2k$, 
the inclusion $\Grass(E) \hookrightarrow \Grass_{\infty}(E)$
is $k$-connected.
Moreover, for $k=1$ the inclusion is an isomorphism on $\pi_1$.
\end{corollary}

\subsection{Classifying spaces for trivializations of exact forms}
\label{subsection:class_spaces_trivs}

Next, we introduce classifying spaces for trivializations of exact forms.
In this section, we give an outline of the construction and state its main properties.
The precise construction and the proof of the properties are relegated to Appendix \ref{appendix:fibrations}.

For a manifold $M$ equipped with a closed $q$-form $\alpha$,
we shall construct a fibration $E^{\alpha} \to M$ with fiber $K(\R,q-1)$.
This fibration has 
the following property: for a manifold $L$ and a map $f:L \to M$,
there is a natural bijection between
trivializations of $f^*\alpha$
and homotopy classes of lifts of $f$ along the fibration $E^{\alpha} \to M$.

The basic idea is a familiar construction in algebraic topology:
the space $E^{\alpha}$ should be obtained from $M$ by
``killing'' the cohomology class of $\alpha$.
Namely, we
consider
a map $a:M \to K(\R,q)$
that represents the cohomology class $[\alpha] \in H^q(M; \R)$
and
define $E^{\alpha} \to M$
as the pullback by $a$ of the path fibration over $K(\R,q)$.
Note that this is the usual model for the homotopy fiber of $f$.

An issue with this idea is that the map $a$ is only defined up to homotopy, 
so that the fibration $E^{\alpha}$
is only defined up to (non-canonical) equivalence.
Moreover,
this construction is not manifestly functorial with respect to maps
$(M, \alpha) \to (N , \beta)$,
since homotopy-commutative diagrams do not induce well-defined maps on homotopy fibers.
We shall define a more precise model for the fibration $E^{\alpha}$
that overcomes these issues.

We now give a precise statement for the outcome of this construction.
For a fixed $q \geq 0$, 
consider the category $\Man$ 
whose objects are pairs $(M, \alpha)$, 
where $M$ is a manifold (possibly with boundary) and 
$\alpha$ is a closed $q$-form on $M$.
A morphism $(M, \alpha) \to (N, \beta)$ in $\Man$
is a smooth map $f: M \to N$ such that $f^*\beta = \alpha$.
We shall construct a functor $\Man \to \Top^{\Arr}$,
where $\Top^{\Arr}$ is the morphism category of $\Top$.
This functor will 
associate to
each object $(M, \alpha)$
a fibration $\vphi_{\alpha}: E^{\alpha} \to M$,
and to each morphism $(M, \alpha) \to (N, \beta)$ given by a map $f:M \to N$
a diagram
\begin{equation}
\label{eq:homotopy_pullback_square}
\begin{tikzcd}
E^{\alpha} \arrow{r} \arrow{d}[swap]{\vphi_{\alpha}}	&E^{\beta} \arrow{d}{\vphi_{\beta}} \\
M \arrow{r}{f}											& N
\end{tikzcd}
\end{equation}
The properties of this construction that we shall need are
summarized in the following proposition.
In the statement, we shall use the following notation.
For a fibration $E \to M$,
we denote by $\Sect(E)$ the set of homotopy classes of sections of $E$.
Moreover, for a map $f: L \to M$, we denote
by $\Lifts_f(E)$ the set of homotopy classes of lifts of $f$ to $E$.

\begin{proposition}
\label{prop:properties_of_fibration}
\hfill
\begin{enumerate}[label = (\roman*)]	 
	
	\item The fibration $E^{\alpha} \to M$
	is weakly equivalent to the homotopy fiber of
	a map $M \to K(\R, q)$ that classifies $[\alpha] \in H^q(M; \R)$.
	
	\item For a pair $(M, \alpha)$ such that $\alpha$ is exact,
	there is a natural bijection $\Triv(\alpha) \iso \Sect(E^{\alpha})$.
	
	\item For a pair $(M, \alpha)$
	and a smooth map $f: L \to M$ 
	such that $f^*\alpha$ is exact, there is a natural bijection
	 \[
	 \begin{tikzcd}
	 \Triv(f^*\alpha) \arrow{r}{\iso} & \Lifts_f(E^{\alpha}).
	 \end{tikzcd}
	 \]

\end{enumerate}
\end{proposition}

The construction of the fibration and the proof of Proposition \ref{prop:properties_of_fibration}
are technical, but straightforward.
For this reason, they are relegated to Appendix \ref{appendix:fibrations}.
For the purpose of proving our main result, we shall only make use of the properties
stated in Proposition \ref{prop:properties_of_fibration}; 
the specific details of the construction will not be of importance to us.

\subsection{Definition of the stable vector bundle $\gamma$}
\label{subsection:def_gamma}

Suppose now that $(M, \omega)$ is a symplectic manifold of dimension $2n$.
For $k \geq n$, consider the stabilized symplectic vector bundle $\C^{k-n} \oplus TM$.
To simplify the notation, we shall write $\Grass_k = \Grass( \C^{k-n} \oplus TM )$.
Let 
$j_k: \Grass_k \to \Grass_{k+1}$
be the stabilization map and let
$\Grass_{\infty} = \colim_k \Grass_k$ be the stable Lagrangian Grassmannian of $M$.
Let $\gamma_{k} \to \Grass_{k}$ be the tautological bundle.
There is an obvious bundle map $\R \oplus \gamma_k \to \gamma_{k+1}$ covering
$j_k$, which yields a pullback diagram:
\[
\begin{tikzcd}
\R \oplus \gamma_{k} \arrow{r}\arrow{d} \arrow[dr, phantom, "\pb"]		&\gamma_{k+1} \arrow{d} \\
\Grass_{k} \arrow{r}[swap]{j_k}											&\Grass_{k+1}.
\end{tikzcd}
\]
We define 
a stable vector bundle $\gamma: \Grass_{\infty} \to BSO$
by the construction described in Section \ref{prelim:construction_stable_bundle}.
To do this, we must choose compatible classifying maps $\Grass_k \to BSO(k)$ for the tautological bundles $\gamma_k \to \Grass_k$.
An additional complication is that we would like
the map $\gamma: \Grass_{\infty} \to BSO$ to be canonically defined up
to a contractible choice.\footnote{This ensures that the
associated Thom spectrum is well-defined up to a contractible choice.} 
We now explain how to achieve this.

Let $\zeta_n \to BU(n)$
be the universal vector bundle.
This bundle has a canonical Hermitian structure.
Let $\Gamma_n \to BU(n)$ be the oriented Lagrangian Grassmannian
of $\zeta_n$.
This is a model for the universal bundle over $BU(n)$
with fiber $U(n)/SO(n)$.
Therefore, $\Gamma_n$ is a model for $BSO(n)$.
Henceforth, we shall use $\Gamma_n$ as our preferred model for $BSO(n)$
and simply write $\Gamma_n = BSO(n)$.

Choose an almost complex structure $J$ on $M$ compatible with $\omega$,
so that $TM$ becomes a Hermitian bundle.
Note that the space of such choices is contractible.
Choose a classifying map
\[
\begin{tikzcd}
TM \arrow{r} \arrow{d}	&\zeta_n \arrow{d} \\
M \arrow{r}				&BU(n)
\end{tikzcd}
\]
for $TM$. Note that this is also a contractible choice.
The classifying map induces a map of the Lagrangian Grassmannians $\Grass(TM) \to \Gamma_n$
which pulls back the tautological bundle over $\Gamma_n$ to the
tautological bundle over $\Grass(TM)$.
Hence $\Grass(TM) \to \Gamma_n = BSO(n)$ is a classifying map for 
the tautological bundle $\gamma_n \to \Grass(TM)$.

The above discussion extends naturally to the stabilized bundles
$\C^{k-n} \oplus TM$.
The choice of classifying map for $TM$ determines
a classifying map for $\C^{k-n} \oplus TM$.
Thus, as above we obtain classifying maps
\[
\Grass_k \to \Gamma_k = BSO(k)
\]
for the tautological bundles $\gamma_k$,
which are compatible with the stabilization maps
$\Grass_k \to \Grass_{k+1}$.
Passing to the colimit, we obtain a stable vector bundle $\gamma: \Grass_{\infty} \to BSO$.

Note that by definition we have a map of fibrations:
\[
\begin{tikzcd}
\Grass_{\infty} \arrow{r}{\gamma} \arrow{d}	 	&BSO \arrow{d} \\
M \arrow{r}							 			&BU
\end{tikzcd}
\]
where the map $M \to BU$
classifies the stable tangent bundle of $M$ and $BSO \to BU$ is the universal bundle with fiber $\Lambda_{\infty} \iso U/SO$.

\subsection{Definition of the stable vector bundle $\gamma^{\omega}$}
\label{subsection:def_gamma_omega}

The stable vector bundle $\gamma^{\omega}$
that we are interested in is obtained from $\gamma$ by adding the data of trivializations of the pullback of $\omega$,
i.e. by ``killing'' the cohomology class of $\omega$.
To construct $\gamma^{\omega}$,
consider the fibration $E^{\omega} \to M$ defined in Section \ref{subsection:class_spaces_trivs}.
Let $\pi : \Grass_{\infty} \to M$ be the projection
and define a fibration $\Grass^{\omega}_{\infty} \to \Grass_{\infty}$ as the following pullback:
\begin{equation}
\label{eq:def_grass_omega}
\begin{tikzcd}
\Grass^{\omega}_{\infty} \arrow{r}\arrow{d}  \arrow[dr, phantom, "\pb"]	& E^{\omega} \arrow{d}{\vphi_{\omega}} \\
\Grass_{\infty} \arrow{r}[swap]{\pi}										&M
\end{tikzcd}
\end{equation}
Finally, we define the stable bundle $\gamma^{\omega}: \Grass^{\omega}_{\infty} \to BSO$
as the composition
\[
\begin{tikzcd}
\Grass^{\omega}_{\infty} \arrow{r} & \Grass_{\infty} \arrow{r}{\gamma} & BSO.
\end{tikzcd}
\]

\section{Proof of the main theorem}
\label{section:proof_main_thm}

In this section, we give the proof of Theorem \ref{thm:main_thm}.
In Section \ref{subsection:iso_stable_group},
we construct a natural map $\Gform(M) \to \Omega_n( \gamma^{\omega})$
and show that it is an isomorphism by a stabilization argument.
Then, in Section \ref{subsection:end_of_proof},
we finish the proof of Theorem \ref{thm:main_thm}
by an application of the Pontrjagin-Thom Theorem.

\subsection{Proof of the isomorphism $\Gform(M) \iso \Omega_n( \gamma^{\omega})$}
\label{subsection:iso_stable_group}

Let $\Omega_n(\gamma^{\omega})$
be the $n$-th cobordism group of the stable vector bundle $\gamma^{\omega}$
constructed in Section \ref{subsection:def_gamma_omega}.
We have a natural map $\sigma: \Gform(M) \to \Omega_n(\gamma^{\omega})$
defined as follows.

Suppose that $\vphi: TL \to TM$ is a formal Lagrangian immersion covering a map $f: L \to M$,
equipped with a trivialization of $f^*\omega$.
By Lemma \ref{lemma:univ_prop_lag_grass},
the bundle map $\vphi$
corresponds naturally to a fiberwise isomorphism $TL \to \gamma_n$,
where $\gamma_n \to \Grass(M)$ is the tautological bundle.
As explained in Section \ref{prelim:construction_stable_bundle},
the map $TL \to \gamma_n$
naturally induces a $\gamma$-structure on $L$.

Moreover, by Proposition \ref{prop:properties_of_fibration}(iii),
the given trivialization of $f^*\omega$ determines a homotopy class of lifts of $f$ along
the fibration $E^{\omega} \to M$.
As $\Grass^{\omega}_{\infty}$ is defined as a pullback \eqref{eq:def_grass_omega},
the $\gamma$-structure on $L$ and
the lift to $E^{\omega}$ together
determine a canonical $\gamma^{\omega}$-structure on $L$.

Applying the same argument to cobordisms,
we conclude that there is a canonical map $\sigma: \Gform(M) \to \Omega_n(\gamma^{\omega})$.

\begin{proposition}
\label{prop:iso_tangential_cob_group}
The map $\sigma:\Gform(M) \to \Omega_n(\gamma^{\omega})$ is an isomorphism.
\end{proposition}
\begin{proof}
Suppose that $L$ is a manifold of dimension $n$ equipped with a 
$\gamma^{\omega}$-structure.
Represent this structure by a map $L \to \Grass^{\omega}_{\infty}$.
By composing with the projections in Diagram \eqref{eq:def_grass_omega}, we obtain the following data:
\begin{itemize}
	\item a map $f: L \to M$, which we may assume to be smooth,
	\item a lift of $f$ to the fibration $E^{\omega} \to M$,
	\item a $\gamma$-structure on $L$ covering $f$.
\end{itemize}

By Proposition \ref{prop:properties_of_fibration}(iii),
the lift of $f$ to $E^{\omega}$
determines a trivialization of $f^*\omega$.
Moreover,
as explained in Section \ref{prelim:construction_stable_bundle},
the $\gamma$-structure on $L$ is represented by a fiberwise isomorphism of vector bundles
\begin{equation}
\label{eq:bundle_map_to_tauto_bundle}
\R^{k-n} \oplus TL \to \gamma_k
\end{equation}
for a large enough $k$, where $\gamma_k \to \Grass_k$ is the tautological bundle.
By Lemma \ref{lemma:stab_homotopy_Lag_grass},
the stabilization map $\Grass(TM) \hookrightarrow \Grass_k$
is $n$-connected.
Since $L$ has dimension $n$, we deduce that the map \eqref{eq:bundle_map_to_tauto_bundle}
factors up to homotopy through a fiberwise isomorphism $TL \to \gamma_n$.
By composing with the canonical map $\gamma_n \to TM$, we obtain a Lagrangian bundle map $TL \to TM$ covering a map $g: L \to M$
homotopic to $f$.

The trivialization of $f^*\omega$ and the homotopy $f \sim g$ induces
a trivialization of $g^*\omega$.
Hence, we have constructed a formal Lagrangian immersion covering the map $g:L \to M$.
It is easy to see that this formal Lagrangian immersion maps to the original $\gamma^{\omega}$-structure
by the stabilization map $\sigma$.
This shows that $\sigma$ is surjective.

To prove that $\sigma$ is injective, apply the same stabilization argument to cobordisms.

\end{proof}

\subsection{End of the proof of the main theorem}
\label{subsection:end_of_proof}

\begin{proof}[Proof of Theorem \ref{thm:main_thm}]
By the Pontrjagin-Thom Theorem for tangential structures (Theorem \ref{thm:PT_tangential}), we have
$
\Omega_n(\gamma^{\omega}) \iso \pi_n \Th( - \gamma^{\omega}).
$
Combining this with the results of the previous sections, we then have a chain of
natural isomorphisms:
\[
\begin{tikzcd}[column sep = large, labels = {inner sep = 1ex}]
\Glag(M) 
\arrow{r}{\iso}[swap]{\text{Prop. \ref{prop:iso_formal_cob_group}}} 
&\Gform(M)
\arrow{r}{\iso}[swap]{\text{Prop \ref{prop:iso_tangential_cob_group}}} 
&\Omega_n(\gamma^{\omega})
\arrow{r}{\iso} 	
&\pi_n \Th(-\gamma^{\omega})
\end{tikzcd}
\]
\end{proof}

\subsection{Identification of $-\gamma^{\omega}$}

In this section, we give a more concrete description of the bundles $-\gamma$ and $-\gamma^{\omega}$.

\begin{proposition}
There is an equivalence of unoriented stable vector bundles over $\Grass_{\infty}(M)$
\[
-\gamma \iso \gamma \oplus \pi^*\nu_M,
\]
where $\nu_M$ is the stable normal bundle of $M$ and $\pi: \Grass_{\infty}(M) \to M$ is the projection.

\end{proposition}
\begin{proof}
Suppose first that $E \to B$ is a symplectic vector bundle,
$\pi: \Grass(E) \to B$ is the oriented Lagrangian Grassmannian
and $\gamma \to \Grass(E)$ is the tautological bundle.
Then $\gamma$ is naturally a Lagrangian subbundle of $\pi^* E$. Hence, there is an isomorphism of complex vector bundles 
$\gamma \tens \C \iso \pi^*E$, where $E$ is equipped with a compatible complex structure. 
Therefore, there is an isomorphism of (unoriented) real vector bundles $\gamma \oplus \gamma \iso \pi^*E$.

We apply this to the stabilized bundles $E_k = \C^{k-n} \oplus TM$.
We then get isomorphisms $\gamma_k \oplus \gamma_k \iso \pi_k^*TM \oplus \C^{k-n}$ of real bundles
over $\Grass_k(M)$,
which are compatible with stabilization.
By passing to the colimit, we get an equivalence of stable vector bundles
$\gamma \oplus \gamma \iso \pi^*TM$,
where $\pi: \Grass_{\infty}(M) \to M$ is the projection.
Taking the inverse bundle, we get an equivalence
\[
-\gamma \iso \gamma - \pi^*TM \iso \gamma + \pi^*\nu_M.
\]

\end{proof}

As a consequence, by taking the pullback to $\Grass_{\infty}^{\omega}(M)$ we obtain an equivalence
\[
-\gamma^{\omega} \iso \gamma^{\omega} \oplus \wtilde{\pi}^*\nu_M,
\]
where $\wtilde{\pi}: \Grass^{\omega}_{\infty}(M) \to M$ is the projection.

\section{Computations}
\label{section:computations}

In this section, we carry out the computations whose results were
stated in Section \ref{subsection:intro_computations}.
In Section \ref{subsection:pi_1},
we compute $\pi_1 \Th( - \gamma^{\omega})$
for every symplectic manifold.
In Section \ref{subsection:monotone},
we consider the case of monotone symplectic manifolds
and give the proof of Proposition \ref{prop:monotone_case}.
Finally, in Section \ref{subsection:surfaces},
we finish the computation of $\Glag(\Sigma)$ for a symplectic surface $\Sigma$ (Theorem \ref{thm:computation_surface}).

\subsection{Computation of $\pi_1$}
\label{subsection:pi_1}

Let $\xi: B \to BSO$ be a stable vector bundle.
By the Thom isomorphism theorem, the spectrum $\Th (\xi)$ is connective, i.e. $\pi_k \Th(\xi) = 0$ for $k < 0$.
Moreover, if $B$ is path-connected then $\pi_0 \Th(\xi) \iso \Z$.

For the computation of $\pi_1 \Th ( -\gamma^{\omega})$, we
consider stable vector bundles satisfying the following assumption.

\begin{assumption}
\label{assumption:nontrivial_map_on_H2}
$B$ is path-connected and
$\xi:B \to BSO$ induces a non-zero map
$H_2(B; \Z) \to H_2( BSO; \Z) \iso \Z_2$.
\end{assumption}

\begin{proposition}
\label{prop:computation_pi1}
Let $\xi:B \to BSO$ be a stable vector bundle satisfying Assumption \ref{assumption:nontrivial_map_on_H2}.
Then the Hurewicz-Thom map 
$\pi_1 \Th( \xi ) \to H_1(B; \Z)$
is an isomorphism.
\end{proposition}
\begin{proof}
We compute $\pi_1 \Th(\xi)$ using the Cartan-Serre method.
Let $\kappa: \Th(\xi) \to MSO$ be the map induced by $\xi: B \to BSO$.
This map preserves the Thom classes, hence it induces an isomorphism on $\pi_0$.
Let $F$ be the homotopy fiber of $\kappa$, so that there is a fiber sequence
\begin{equation}
\label{eq:fiber_seq}
\begin{tikzcd}
F \arrow{r} &\Th(\xi) \arrow{r}{\kappa} &MSO.
\end{tikzcd}
\end{equation}
Consider the first few terms of the long exact sequence of homotopy groups:
\[
\begin{tikzcd}[column sep = small]
\pi_2 MSO \arrow{r}	&\pi_1 F \arrow{r} &\pi_1 \Th( \xi) \arrow{r} &\pi_1 MSO \arrow{r} &\pi_0 F \arrow{r} &\pi_0 \Th ( \xi) \arrow{r} &\pi_0 MSO
\end{tikzcd}
\]
We have $\pi_2 MSO = \pi_1 MSO = 0$ and the map $\pi_0 \Th(\xi) \to \pi_0 MSO$ is an isomorphism. 
Hence we conclude that $\pi_0 F = 0$
and that $\pi_1 F \to \pi_1 \Th(\xi)$ is an isomorphism.

Since $\pi_0 F = 0$, the Hurewicz map $\pi_1 F \to H_1(F; \Z)$ is an isomorphism.
Therefore, to prove the proposition, it remains to show that the map
$H_1(F; \Z) \to H_1( \Th(\xi); \Z)$ is an isomorphism.
To see this, consider the long exact sequence in homology induced by \eqref{eq:fiber_seq}
\[
\begin{tikzcd}[column sep = small]
H_2( \Th(\xi); \Z) \arrow{r} & H_2(MSO; \Z) \arrow{r} &H_1(F; \Z) \arrow{r} &H_1(\Th(\xi); \Z) \arrow{r} &H_1(MSO; \Z).
\end{tikzcd}
\]
We have $H_1(MSO; \Z) \iso H_1( BSO; \Z)= 0$ since $BSO$ is simply-connected.
Moreover, $H_2(MSO; \Z) \iso H_2(BSO; \Z) \iso \Z_2$.
Finally, by Assumption \ref{assumption:nontrivial_map_on_H2},
the map $H_2( \Th(\xi); \Z) \to H_2(MSO; \Z)$ is surjective.
Hence, we conclude that $H_1(F; \Z) \to H_1( \Th(\xi); \Z)$ is an isomorphism.
\end{proof}

To apply Proposition \ref{prop:computation_pi1} to the bundle $-\gamma^{\omega}:\Grass_{\infty}^{\omega} \to BSO$,
we must check that it satisfies Assumption \ref{assumption:nontrivial_map_on_H2}.

\begin{proposition}
The bundle $-\gamma^{\omega}$ satisfies Assumption \ref{assumption:nontrivial_map_on_H2}.
\end{proposition}
\begin{proof}
Assumption \ref{assumption:nontrivial_map_on_H2} is invariant under
replacing a bundle $\xi$ by its inverse $- \xi$, since the
classifying maps of $\xi$ and $-\xi$
differ by a homotopy equivalence of $BSO$.
Hence, it
suffices to prove the assertion for $\gamma^{\omega}$.

Recall from Section \ref{subsection:def_gamma} 
that $\Grass_{\infty} \to M$ and $BSO \to BU$ are fibrations with fiber $\Lambda_{\infty}$.
Moreover, the classifying map $\gamma: \Grass_{\infty} \to BSO$
fits into a map of fibrations
\[
\begin{tikzcd}
\Grass_{\infty} \arrow{d} \arrow{r}{\gamma}	& BSO \arrow{d} \\
M \arrow{r}									&BU
\end{tikzcd}
\]
where $M \to BU$ classifies the stable tangent bundle of $M$.
After passing to the pullback by $E^{\omega} \to M$,
we obtain a map of fiber sequences
\begin{equation}
\label{cd:map_fiber_seq}
\begin{tikzcd}
\Lambda_{\infty} \times K(\R,1)  \arrow{d}{\proj} \arrow{r} &\Grass_{\infty}^{\omega} \arrow{r} \arrow{d}{\gamma^{\omega}} & M \arrow{d} \\
\Lambda_{\infty} \arrow{r}				 			& BSO \arrow{r}					 								&BU
\end{tikzcd}
\end{equation}
From Diagram \eqref{cd:map_fiber_seq},
we deduce that it suffices to prove that 
the map $\Lambda_{\infty} \to BSO$ is non-trivial on $H_2(- ; \Z)$.
To see this, note first that the map $\Lambda_{\infty} \to BSO$
pulls back the universal Stiefel-Whitney class $w_2 \in H^2(BSO; \Z_2)$
to the generator of $H^2(\Lambda_{\infty}; \Z_2) \iso \Z_2$; see e.g. \cite[1.6]{Audin-calculs}.
Moreover, $BSO$ is simply-connected
and $H_1(\Lambda_{\infty}; \Z) \iso \Z$,
so that in both cases there is no $\Ext$ term
in the universal coefficients sequence for $H^2(-; \Z_2)$.
Therefore, the non-vanishing of the map on $H^2(- ; \Z_2)$
implies the non-vanishing of the map on $H_2(-; \Z)$.
\end{proof}

\begin{corollary}
\label{cor:Hurewicz_map_is_an_iso}
The Hurewicz-Thom map $\pi_1 \Th ( - \gamma^{\omega} ) \to H_1( \Grass_{\infty}^{\omega}; \Z)$
is an isomorphism.
\end{corollary}

The homology $H_1( \Grass_{\infty}^{\omega}; \Z)$ can
be further computed using the Serre spectral sequence for the fibration $\Grass_{\infty}^{\omega} \to \Grass_{\infty}$
with fiber $K(\R,1)$.
Before doing this, we make the following definition.
Let $\pi: \Grass_{\infty}  \to M$
be the projection.
Consider the class $\pi^*[\omega] \in H^2(\Grass_{\infty}  ; \R)$.
Note that this is the characteristic class of the fibration $\Grass_{\infty}^{\omega} \to \Grass_{\infty}$.
We define the stable periods $\Per \subset \R$ of $\omega$ to be the image of the
evaluation map
\begin{equation}
\label{eq:evaluation_omega}
\pi^*[\omega]: H_2( \Grass_{\infty}  ; \Z) \to \R.
\end{equation}

\begin{proposition}
\label{prop:computation_H1}
$H_1(\Grass_{\infty}^{\omega}; \Z) \iso H_1( \Grass_{\infty}; \Z) \oplus \R/\Per$.
\end{proposition}
\begin{proof}
Consider the $E^2$ page of the
Serre spectral sequence of the fibration $\Grass_{\infty}^{\omega} \to \Grass_{\infty}$.
Note that this fibration is orientable since it is pulled back from a fibration over $K(\R,2)$.
The only relevant differential for the computation of $H_1$
is the transgression $d^2_{2,0}: E^2_{2,0} \to E^2_{0,1}$.
We have $E^2_{2,0} \iso H_2( \Grass_{\infty}; \Z)$
and $E^2_{0,1} \iso H_1(K(\R,1); \Z) \iso \R$.
Moreover, the transgression $d^2_{2,0}$ is identified
with the evaluation of the characteristic class, i.e. with the map \eqref{eq:evaluation_omega}.

It follows that the image of $d^2_{2,0}$ is identified with the stable periods $\Per$.
The Serre spectral sequence then yields an exact sequence
\[
\begin{tikzcd}
0 \arrow{r} &\R/\Per \arrow{r} & H_1( \Grass_{\infty}^{\omega} ; \Z) \arrow{r}	&H_1( \Grass_{\infty}; \Z) \arrow{r} &0.
\end{tikzcd}
\]
Since $\R/\Per$ is divisible,
this exact sequence splits and we conclude that $H_1( \Grass_{\infty}^{\omega} ; \Z) \iso H_1( \Grass_{\infty}; \Z) \oplus \R/\Per$.
\end{proof}

\subsection{Monotone symplectic manifolds}
\label{subsection:monotone}

Recall that we denote by $\Grass_k M = \Grass(\C^{k-n} \oplus TM)$
the stabilized Lagrangian Grassmannians
and by $\pi_k: \Grass_k M \to M$
the projections.

\begin{lemma}
\label{lemma:pullback_exact_for_monotone_manifold}
Let $(M, \omega)$ be a monotone manifold.
Then for every $k \geq n$,
the form $\pi_k^* \omega$ is exact.
\end{lemma}
\begin{proof}
The proof is the same as the proof of \cite[Corollary A.3]{DRF23},
which we recall for completeness.

For a symplectic vector bundle $E \to B$,
the tautological bundle $\gamma \to \Grass(E)$
is a Lagrangian subbundle of the pullback $\pi^*E$, where $\pi: \Grass(E) \to B$
is the projection.
This gives an isomorphism of complex vector bundles $\pi^*E \iso \gamma \tens \C$.
In turn, this implies that
\[
\pi^* c_1(E) = c_1( \pi^*E) = c_1 (\gamma \tens \C) = 0,
\]
since $\gamma$ is oriented.

Suppose now that $M$ is a monotone symplectic manifold.
Then $[\omega] = \tau  \, c_1(TM)$, hence we compute:
\[
\pi_k^* [\omega] = \tau \, \pi_k^* c_1(TM)= \tau \, \pi_k^* c_1(\C^{k-n} \oplus TM) = 0.
\]
\end{proof}

We now give the proof of the splitting $\Th(-\gamma^{\omega}) \hmtpyeq \Th(-\gamma) \Smash K(\R,1)_+$ 
for a monotone manifold.

\begin{proof}[Proof of Proposition \ref{prop:monotone_case}]
Let $\pi: \Grass_{\infty} M \to M$
be the projection of the stable Lagrangian Grassmannian.
By Lemma \ref{lemma:pullback_exact_for_monotone_manifold},
we have $\pi^*[\omega] = 0 \in H^2(\Grass_{\infty} M; \R)$.
Therefore,
the pullback by $\pi$ of the fibration $E^{\omega} \to M$
is fiber-homotopically trivial.
Choose a trivialization $\Grass_{\infty}^{\omega}M \hmtpyeq \Grass_{\infty}M \times K(\R,1)$
as fibrations over $\Grass_{\infty} M$.
This induces a homotopy equivalence of stable vector bundles
\[
\begin{tikzcd}
\Grass_{\infty}^{\omega}M \arrow{d}[swap]{\gamma^{\omega}} \arrow{r}{\hmtpyeq}	& \Grass_{\infty} M \times K(\R,1) \arrow{dl}{\gamma \times *} \\
BSO
\end{tikzcd}
\]
Hence,
there is an induced stable equivalence of Thom spectra
\[
\Th(\gamma^{\omega}) \hmtpyeq \Th(\gamma) \Smash K(\R,1)_+.
\]
Taking the inverse bundle, this gives an equivalence $\Th(-\gamma^{\omega}) \hmtpyeq \Th(-\gamma) \Smash K(\R,1)_+$.
\end{proof}

The splitting $\Th(-\gamma^{\omega}) \hmtpyeq \Th(-\gamma) \Smash K(\R,1)_+$,
combined with Theorem \ref{thm:main_thm}, gives an isomorphism
\begin{equation}
\label{eq:iso_homology_K(R,1)}
\Glag(M) \iso H_n( K(\R,1) ; \Th(-\gamma))
\end{equation}
for any monotone manifold of dimension $2n$.
The homology group appearing on the
right-hand side of Equation \eqref{eq:iso_homology_K(R,1)}
can be further computed
using the Atiyah-Hirzebruch spectral sequence.
The following computation is a straightforward generalization
of a result of Audin (see \cite[Théorème 2.1]{Audin-thesis}
or \cite[Theorem 1]{Audin-cotangent}).

\begin{proposition}
\label{prop:iso_tensor_product}
For a monotone symplectic manifold $M$ of dimension $2n$,
there is an isomorphism
\[
\Glag(M) \iso \bigoplus_{q = 0}^n \; H_{n-q}( K(\R,1); \Z) \tens \pi_q \Th(-\gamma).
\]
\end{proposition}
\begin{proof}
The proof is the same as the proof of Theorem 1 of \cite{Audin-cotangent}.
One considers the Atiyah-Hirzebruch spectral sequence
with $E^2$ page 
\[
E^2_{p,q} = H_p( K(\R,1); \pi_q \Th(-\gamma)).
\]
Since the homology of $K(\R,1)$ over
$\Z$ consists of rational vector spaces in positive degrees,
there is no $\Tor$ term in the universal coefficients sequence, so that
\[
H_p( K(\R,1); \pi_q \Th(-\gamma)) \iso H_p(K(\R,1); \Z) \tens_{\Z} \pi_q \Th(-\gamma).
\]
Moreover, since the Atiyah-Hirzebruch spectral sequence is trivial over $\Q$,
the spectral sequence collapses at the $E^2$ page and the result follows.
\end{proof}

\subsection{Surfaces}
\label{subsection:surfaces}

\begin{lemma}
\label{lemma:computation_periods_surfaces}
For a closed symplectic surface $(\Sigma, \omega)$ of genus $g$,
we have
\[
\R/\Per \iso 
\begin{cases}
\R &\text{if } g \neq 1, \\
\R/\Z &\text{if } g = 1.
\end{cases}
\]
\end{lemma}
\begin{proof}
If $g \neq 1$, then $c_1(\Sigma) \neq 0$,
so that
$\Sigma$ is monotone. Hence
$\Per = 0$ by Lemma \ref{lemma:pullback_exact_for_monotone_manifold}.

Suppose now that $g = 1$, so that $\Sigma \iso T^2$.
Since $T T^2$ is a trivial bundle,
the Lagrangian Grassmannian bundle $\Grass T^2 \to T^2$ and its stabilizations
are trivial bundles.
Passing to the colimit, we have $\Grass_{\infty} T^2 \iso T^2 \times \Lambda_{\infty}$
as bundles over $T^2$.
This implies that $\pi$ is surjective on $H_2(-; \Z)$, hence the periods
of $\pi^*\omega$ coincide with the periods of $\omega$.
Therefore, we conclude that
\[
\Per = \Image( \omega: H_2(T^2; \Z) \to \R ) = A \Z,
\]
where $A = \int_{T^2} \omega$.
\end{proof}

\begin{proof}[Proof of Theorem \ref{thm:computation_surface}]
By Corollary \ref{cor:Hurewicz_map_is_an_iso} and Proposition \ref{prop:computation_H1},
for a symplectic surface $(\Sigma, \omega)$
we have
\[
\Glag(\Sigma) \iso H_1(\Grass_{\infty} \Sigma; \Z) \oplus \R/\Per,
\]
where the term $\R/\Per$ is computed in Lemma \ref{lemma:computation_periods_surfaces}.
To finish the proof, we observe
that the inclusion $\Grass \Sigma \hookrightarrow \Grass_{\infty} \Sigma$
is an isomorphism on $\pi_1$ by Corollary \ref{cor:inclusion_in_colimit_is_k_connected}, and therefore is an isomorphism on $H_1(-; \Z)$.
\end{proof}

\appendix

\section{Classifying fibrations for trivializations of forms}
\label{appendix:fibrations}

In this appendix, we
construct the classifying fibrations $E^{\alpha} \to M$
whose existence was asserted in Section \ref{subsection:class_spaces_trivs},
and give the proof of Proposition \ref{prop:properties_of_fibration}.
The construction is based on simplicial methods.
The basic ingredient is the fact that, in the world of simplicial sets,
the representation theorem for cohomology
holds at the cochain level.
This allows us to define classifying maps for 
closed forms that are well-defined up to a canonical homotopy equivalence.\footnote{
It is plausible that this method in fact yields classifying maps that
are canonical up to a contractible choice, but we shall not prove this.}

\subsection{A simplicial model}
\label{subsection:simplicial_stuff}

We start by defining a convenient simplicial model for the path fibration over $K(\R,q)$.
In the following, we shall make use of some well-known facts about simplicial sets;
standard references for this material are \cite{May, Goerss-Jardine}.

For a simplicial set $X$, let $C^*(X; \R)$
be the complex of normalized $\R$-valued simplicial cochains on $X$.
Denote by $Z^*(X; \R)$ the cocycles of this complex.

Following \cite[\S{}23]{May},
we consider the simplicial abelian group $C(\R,q)$ whose set of 
$n$-simplices is $C^q( \Delta^n; \R)$,
where $\Delta^n$ is the standard $n$-simplex.
Let $Z(\R,q)$ the simplicial subgroup of $C(\R, q)$
given by the cocycles.
The differential of the normalized cochain complex defines a simplicial map $d: C(\R, q-1) \to Z(\R, q)$
whose kernel is $Z(\R, q-1)$.

We shall need the following
well-known properties of these simplicial sets,
whose proofs are given in \cite{May}.
\begin{enumerate}[label = (\roman*)]
\item The geometric realization $|Z(\R,q)|$ is an
Eilenberg-MacLane space $K(\R, q)$.
\item $C(\R, q)$ is contractible.
\item $d$ is a Kan fibration between Kan complexes.
\end{enumerate}

It follows from these properties that the geometric realization $|d|: |C(\R, q-1)| \to |Z(\R, q)|$
is a Serre fibration 
and a model for the path fibration over $K(\R,q)$.
From now on, we shall use this fibration as our preferred model
for the path fibration. 
To simplify the notation, we will write 
$K(\R,q) = | Z(\R,q)|$,
$P(\R,q) = | C(\R,q-1)|$
and 
$\delta = |d|$.

The simplicial set $C(\R,q)$ represents the normalized cochains functor in the sense
that for any simplicial set $X$ there is a natural isomorphism $C^q(X; \R) \iso  \Hom_{\sSet}( X, C(\R,q))$,
see \cite[Lemma 24.2]{May}.
Similarly, there is a natural isomorphism $Z^q(X; \R) \iso \Hom_{\sSet}(X, Z(\R, q))$,
see \cite[Lemma 24.3]{May}.

If $\alpha$ is a $q$-cocycle on $X$ represented by the map $a: X \to Z(\R,q)$,
then in the following diagram:
\begin{equation}
\label{eq:lifts_simplicial_case}
\begin{tikzcd}
											&C(\R, q-1) \arrow{d}{d} \\
X \arrow{r}[swap]{a} \arrow[dashed]{ur}		&Z(\R,q)
\end{tikzcd}
\end{equation}
a lift of $a$ exists if and only if $\alpha$ is exact.
Moreover, in this case the set of lifts is naturally in bijection with the set of primitives of $\alpha$.
It follows from
\cite[Theorem 24.4]{May}
that this descends to a
bijection between the set of homotopy classes
of lifts of $a$ and the set $\Triv(\alpha)$ of cohomology classes of primitives of $\alpha$.

\subsection{Construction of $E^{\alpha}$}

Suppose now that $M$ is a smooth manifold
and let $\SingSm M$ be the simplicial set of smooth singular simplices in $M$.
There is a natural chain map
\begin{equation}
\label{eq:map_from_deRham_cochains_to_simplicial_cochains}
\Omega^*(M) \to C^*(\SingSm M; \R),
\end{equation}
that assigns to a differential form its integral on simplices.
This chain map is a quasi-isomorphism by the de Rham Theorem.

Let $\alpha$ a closed $q$-form on $M$.
By the universal property of $Z(\R,q)$, the normalized cocycle that corresponds to $\alpha$
by the chain map \eqref{eq:map_from_deRham_cochains_to_simplicial_cochains}
is represented by a map
$a: \SingSm M \to Z(\R,q)$.
Taking the geometric realization
yields a map $|a|: |\SingSm M| \to |Z(\R,q)|$.
Define a Serre fibration $\psi_{\alpha}: F^{\alpha} \to |\SingSm M|$ as the following pullback:
\begin{equation}
\label{eq:def_psi}
\begin{tikzcd}
F^{\alpha} \arrow{r} \arrow{d}[swap]{\psi_{\alpha}}	\arrow[dr, phantom, "\pb"]	& P(\R,q) \arrow{d}{\delta} \\
\left|\SingSm M \right| \arrow{r}[swap]{\left|a\right|}						& K(\R,q)
\end{tikzcd}
\end{equation}

Note that if $(M,\alpha) \to (N, \beta)$ is a morphism corresponding to $f:M \to N$,
then by naturality the following diagram commutes:
\[
\begin{tikzcd}
\left|\SingSm M\right| \arrow{d}[swap]{f_*} \arrow{r}{\left| a \right|} 	&\left| Z(\R,q) \right|		\\
\left|\SingSm N\right| \arrow{ur}[swap]{\left| b \right|}
\end{tikzcd}
\]
where $a$ and $b$ are the maps representing $\alpha$ and $\beta$, respectively.
Hence there is an induced pullback diagram
\begin{equation}
\label{eq:homotopy_pullback}
\begin{tikzcd}
F^{\alpha} \arrow{d}[swap]{\psi_{\alpha}} \arrow{r}	\arrow[dr, phantom, "\pb"]  & F^{\beta} \arrow{d}{\psi_{\beta}} \\
\left| \SingSm M \right| \arrow{r}[swap]{f_*}									&\left| \SingSm N \right|
\end{tikzcd}
\end{equation}

To finish the construction, we replace the fibration $\psi_{\alpha}$ by a fibration over $M$.
Consider the natural map $\eps^{\infty}: |\SingSm M| \to M$
obtained as the composition
\[
\begin{tikzcd}
\left| \SingSm M \right| \arrow{r}{\inc} & \left|\Sing M\right| \arrow{r}{\eps} & M,
\end{tikzcd}
\]
where $\Sing M$ is the singular set of $M$ and $\eps:|\Sing M| \to M$ is the counit of the adjunction
between geometric realization and the singular set functor.
The map $\eps$ is a homotopy equivalence since $M$ has the homotopy type of a CW complex \cite[Theorem 16.6]{May}.
Moreover, the inclusion $\SingSm M \hookrightarrow \Sing M$ is a homotopy equivalence
by the Whitney approximation theorem.
It follows that $\eps^{\infty}$ is also a homotopy equivalence.

Finally, we define $\vphi_{\alpha}: E^{\alpha} \to M$
to be the fibration obtained by the usual functorial replacement of $\eps^{\infty} \circ \psi_{\alpha}$
by a homotopy-equivalent fibration.
Hence, we have the following diagram:

\begin{equation}
\label{eq:equivalence_of_fibrations}
\begin{tikzcd}
F^{\alpha} \arrow{d}[swap]{ \psi_{\alpha} } \arrow{r}{\hmtpyeq}	& E^{\alpha} \arrow{d}{\vphi_{\alpha}} \\
\left| \SingSm M \right| \arrow{r}{\hmtpyeq}[swap]{\eps^{\infty}}						& M
\end{tikzcd}
\end{equation}

All the steps in the above construction are clearly functorial in $(M,\alpha)$,
so we have indeed defined a functor $\Man \to \Top^{\Arr}$.

\subsection{Preparation for the proof of Proposition \ref{prop:properties_of_fibration}}
\label{subsection:preparation}

In this section, we prove a few preparatory lemmas for the proof of Proposition \ref{prop:properties_of_fibration}.

Recall that for any simplicial set $X$, the unit of the adjunction $|\cdot|  \dashv \Sing$
is a natural weak equivalence
\[
\begin{tikzcd}[column sep = small]
\eta: X \arrow{r}{\hmtpyeq} &\Sing |X|.
\end{tikzcd}
\]
Hence $\eta$ induces an isomorphism on cohomology (where we assume $\R$ coefficients throughout):
\[
\eta^*: H^*(|X|)\to H^*(X).
\]

Applying this to $X = Z(\R,q)$, we define the fundamental class of $K(\R,q) = |Z(\R,q)|$
to be the preimage of the fundamental class of $Z(\R,q)$ by the isomorphism $\eta^*$.

Next, applying this to $ X = \Sing M$ for a space $M$, we get a canonical isomorphism
\[
\begin{tikzcd}
\eta^*: H^*(|\Sing M|) \arrow{r}{\iso}	&\underbrace{H^*(\Sing M)}_{=H^*(M)}
\end{tikzcd}
\]
Moreover, the counit $\eps: |\Sing M| \to M$ induces an isomorphism going the other way:
\[
\eps^*: H^*(M) \to H^*(|\Sing M|).
\]
\begin{lemma}
\label{lemma:unit_counit_are_inverses}
The maps $\eta^*$ and $\eps^*$ are inverses.
\end{lemma}
\begin{proof}
This follows from the triangle identity for the unit and counit of an adjunction,
which in the present case states that the following diagram commutes for any space $M$:
\[
\begin{tikzcd}
\Sing M \arrow{r}{\eta_{\Sing M}} \arrow{dr}[swap]{\id}	&\Sing |\Sing M| \arrow{d}{\Sing(\eps)} \\
														&\Sing M
\end{tikzcd}
\]
Taking the induced maps on cohomology yields the claim.
\end{proof}

Suppose now that $\alpha$ is a normalized singular cocycle on $M$.
Let $a: \Sing M \to Z(\R,q)$
be the map representing $\alpha$.
Then, we have a diagram
\begin{equation}
\label{cd:classifying_map_of_cocycle}
\begin{tikzcd}
M 	&\left| \Sing M \right| \arrow{l}[swap]{\eps}{\hmtpyeq}	 \arrow{r}{\left| a \right|}  &K(\R,q).
\end{tikzcd}
\end{equation}
Inverting $\eps$ determines a well-defined homotopy class of maps $M \to K(\R,q)$.

\begin{lemma}
\label{lemma:classifying_map_cohomology_singular_case}
The homotopy class of maps $M \to K(\R,q)$ determined by Diagram \eqref{cd:classifying_map_of_cocycle} represents
$[\alpha]$, in the sense that it pulls back the fundamental class of $K(\R,q)$
to $[\alpha]$.
\end{lemma}
\begin{proof}
By naturality, we have the following diagram:
\[
\begin{tikzcd}
H^*(|\Sing M|) \arrow{d}[swap]{\eta^*}	& H^*(K (\R,q)) \arrow{l}[swap]{|a|^*} \arrow{d}{\eta^*} \\
H^*(M) 									&H^*(Z(\R,q)) \arrow{l}{a^*}
\end{tikzcd}
\]
By definition, $a^* \eta^* \iota = [\alpha]$,
where $\iota \in H^q( K(\R,q) )$ is the fundamental class.
Hence, by the above diagram we get $\eta^* |a|^* \iota = [\alpha]$.
Finally, 
by Lemma \ref{lemma:unit_counit_are_inverses}
we have $\eta^* = (\eps^*)^{-1}$,
which proves the claim.
\end{proof}

\subsection{Proof of Proposition \ref{prop:properties_of_fibration}}
\label{subsection:properties_fibration}

We now prove that the above construction has the desired properties.

\begin{proof}[Proof of Proposition \ref{prop:properties_of_fibration}]
\hfill

(i)
Recall that we have a diagram
\begin{equation}
\label{cd:classifying_map_of_cocycle_smooth_case}
\begin{tikzcd}
M 	&\left| \SingSm M \right| \arrow{l}[swap]{\eps^{\infty}}{\hmtpyeq}	 \arrow{r}{\left| a \right|}  &K(\R,q)
\end{tikzcd}
\end{equation}
where $\eps^{\infty}$ is a homotopy equivalence.
Inverting $\eps^{\infty}$ determines a homotopy class of maps $M \to K(\R,q)$,
and by construction the fibration $E^{\alpha}$
is a representative for the homotopy fiber of the maps in this homotopy class.
Using Lemma \ref{lemma:classifying_map_cohomology_singular_case} 
and the fact that the inclusion $\SingSm M \hookrightarrow \Sing M$ is
a homotopy equivalence,
we deduce that
the homotopy class of maps $M \to K(\R,q)$
determined by Diagram \eqref{cd:classifying_map_of_cocycle_smooth_case}
represents $[\alpha] \in H^q(M; \R)$.
Therefore, the fibration $E^{\alpha}$
is a representative of the homotopy fiber of a classifying map for $[\alpha]$.

(ii)
Let $M$ be a manifold with an exact form $\alpha$.
We shall relate the set of trivializations $\Triv(\alpha)$
and the set of sections of $E^{\alpha}$ by a chain of natural bijections.

Denote by $\overline{\alpha}$ the simplicial cocycle corresponding 
to $\alpha$ by the de Rham quasi-isomorphism \eqref{eq:map_from_deRham_cochains_to_simplicial_cochains}.
Denote by $a: \Sing^{\infty} M \to Z(\R,q)$
the map classifying $\overline{\alpha}$.

The quasi-isomorphism \eqref{eq:map_from_deRham_cochains_to_simplicial_cochains} induces 
a natural bijection between the sets of trivializations $\Triv(\alpha) \iso \Triv( \overline{\alpha})$.
Moreover, as discussed in Section \ref{subsection:simplicial_stuff},
there is a natural bijection between
$\Triv( \overline{\alpha})$
and the set of homotopy classes of lifts of $a$ in the following diagram:
\[
\begin{tikzcd}
															&C(\R, q-1) \arrow{d}{d} \\
\SingSm  M \arrow{r}[swap]{a} \arrow[dashed]{ur}		&Z(\R,q)
\end{tikzcd}
\]
Applying geometric realization, we obtain a bijection between
the homotopy classes of lifts of $a$ and the homotopy classes of lifts of $|a|$ in the following diagram:
\[
\begin{tikzcd}
																				&P(\R,q) \arrow{d}{\delta} \\
\left|\SingSm M\right| \arrow{r}[swap]{\left| a \right|} \arrow[dashed]{ur}		&K(\R,q)
\end{tikzcd}
\]
By definition of $E^{\alpha}$, we have a diagram of fibrations as follows:
\begin{equation}
\label{cd:maps_of_fibrations}
\begin{tikzcd}
E^{\alpha} \arrow{d} 	&F^{\alpha} \arrow{l}{\hmtpyeq} \arrow{d} 		\arrow{r}						& P(\R,q) \arrow{d}{\delta} \\
M 						&\left| \SingSm M \right|  \arrow{l}{\hmtpyeq} 	\arrow{r}{\left| a \right|}		& K(\R,q).
\end{tikzcd}
\end{equation}
The right square of Diagram \eqref{cd:maps_of_fibrations} is a pullback,
so there is a bijection between lifts of $|a|$ and sections of $F^{\alpha}$.
Finally, the left square of Diagram \eqref{cd:maps_of_fibrations} is a homotopy pullback, so it induces a bijection 
$\Sect(E^{\alpha}) \iso \Sect(F^{\alpha})$.

(iii)
Let $f:L \to M$
be a smooth map such that $f^*\alpha$ is exact.
Consider $f$ as a morphism $(L, f^*\alpha) \to (M, \alpha)$.
By functoriality, there is a diagram
\[
\begin{tikzcd}
E^{f^*\alpha} \arrow{r}	\arrow{d}	&E^{\alpha} \arrow{d} \\
L 		\arrow{r}{f}				&M
\end{tikzcd}
\]
This diagram is a homotopy pullback since, by definition, it is pointwise homotopy equivalent to the homotopy
pullback square \eqref{eq:homotopy_pullback}. 
Hence, the induced map $E^{f^*\alpha} \to f^* E^{\alpha}$
is a weak equivalence.
As a consequence, the induced map on the level of homotopy classes of sections gives a bijection
$\Sect(E^{f^*\alpha}) \iso \Sect(f^* E^{\alpha})$.
Moreover, sections of $f^*E^{\alpha}$ correspond naturally to lifts of $f$ to $E^{\alpha}$,
hence we have a bijection
$\Sect(f^*E^{\alpha}) \iso \Lifts_f(E^{\alpha})$.
Finally, since $f^*\alpha$ is exact, by (ii) we have a bijection $\Sect(E^{f^*\alpha}) \iso \Triv(f^*\alpha)$.
This proves the claim.
\end{proof}

\bibliographystyle{alpha}
\bibliography{bibliography}

\begin{thebibliography}{Aud87b}

\bibitem[Arn80]{Arnold}
V.~I. Arnol'd.
\newblock Lagrange and {L}egendre cobordisms. {I}.
\newblock {\em Funktsional. Anal. i Prilozhen.}, 14(3):1--13, 96, 1980.

\bibitem[Aud85]{Audin-calculs}
M.~Audin.
\newblock Quelques calculs en cobordisme lagrangien.
\newblock {\em Ann. Inst. Fourier (Grenoble)}, 35(3):159--194, 1985.

\bibitem[Aud87a]{Audin-thesis}
M.~Audin.
\newblock {\em Cobordismes d'immersions lagrangiennes et legendriennes}.
\newblock Hermann, Paris, 1987.

\bibitem[Aud87b]{Audin-cotangent}
M.~Audin.
\newblock Cobordisms of lagrangian immersions in the space of the cotangent
  bundle of a manifold.
\newblock {\em Funct. Anal. Appl.}, 21:223--226, 1987.

\bibitem[BT82]{Bott-Tu}
R.~Bott and L.~W. Tu.
\newblock {\em Differential forms in algebraic topology}.
\newblock Springer-Verlag, New York-Berlin, 1982.

\bibitem[Eli84]{Eliashberg}
Y.~Eliashberg.
\newblock Cobordisme des solutions de relations diff\'{e}rentielles.
\newblock In {\em South {R}hone seminar on geometry, {I} ({L}yon, 1983)},
  Travaux en Cours, pages 17--31. Hermann, Paris, 1984.

\bibitem[EM02]{Eliashberg-Mishachev}
Y.~Eliashberg and N.~Mishachev.
\newblock {\em Introduction to the {$h$}-principle}.
\newblock American Mathematical Society, Providence, RI, 2002.

\bibitem[GJ09]{Goerss-Jardine}
P.~G. Goerss and J.~F. Jardine.
\newblock {\em Simplicial homotopy theory}.
\newblock Birkh\"{a}user Verlag, Basel, 2009.

\bibitem[Gro71]{Gromov-icm}
M.~Gromov.
\newblock A topological technique for the construction of solutions of
  differential equations and inequalities.
\newblock In {\em Actes du {C}ongr\`es {I}nternational des {M}ath\'{e}maticiens
  ({N}ice, 1970), {T}ome 2}, pages 221--225. Gauthier-Villars \'{E}diteur,
  Paris, 1971.

\bibitem[Gro86]{Gromov-pdr}
M.~Gromov.
\newblock {\em Partial differential relations}.
\newblock Springer-Verlag, Berlin, 1986.

\bibitem[Las63]{Lashof}
R.~Lashof.
\newblock Poincar\'{e} duality and cobordism.
\newblock {\em Trans. Amer. Math. Soc.}, 109:257--277, 1963.

\bibitem[Lee76]{Lees}
J.~A. Lees.
\newblock On the classification of {L}agrange immersions.
\newblock {\em Duke Math. J.}, 43(2):217--224, 1976.

\bibitem[May67]{May}
J.~P. May.
\newblock {\em Simplicial objects in algebraic topology}.
\newblock D. Van Nostrand Co., Inc., Princeton, N.J.-Toronto, Ont.-London,
  1967.

\bibitem[Per19]{Perrier19}
A.~Perrier.
\newblock Lagrangian cobordism groups of higher genus surfaces.
\newblock Preprint, can be found at arXiv:1901.06002, 2019.

\bibitem[RF23]{DRF23}
D.~Rathel-Fournier.
\newblock Unobstructed lagrangian cobordism groups of surfaces.
\newblock Preprint, can be found at arXiv:2307.03124, 2023.

\bibitem[Sto68]{Stong}
R.~E. Stong.
\newblock {\em Notes on cobordism theory}.
\newblock Princeton University Press, Princeton, NJ; University of Tokyo Press,
  Tokyo, 1968.

\end{thebibliography}
	
\end{document}